\documentclass[EJP]{ejpecp} % replace ECP by EJP if needed.
\def\beq{ \begin{equation} }
\def\eeq{ \end{equation} }
\def\mn{\medskip\noindent}

\def\ep{\epsilon}

\def\square{\vcenter{\vbox{\hrule height .4pt
  \hbox{\vrule width .4pt height 5pt \kern 5pt
        \vrule width .4pt} \hrule height .4pt}}}

\def\RR{\mathbb{R}}

\def\FF{\mathcal{F}}

\def\GG{\mathcal{G}}
\def\KK{\mathcal{K}}
\def\EE{\mathcal{E}}
\def\II{\mathcal{I}}

\def\AA{\mathcal{A}}

\def\hh{\hspace{1ex}}
\def\ep{\varepsilon}

\def\hh{\hspace{1ex}}

\numberwithin{equation}{section}

\SHORTTITLE{Degree centrality}

\TITLE{Degree centrality and root finding in growing random networks}
\AUTHORS{%
  Sayan Banerjee\footnote{Department of Statistics and Operations Research, Hanes Hall, University of North Carolina, Chapel Hill, NC 27599.
    \EMAIL{sayan@email.unc.edu}}
  \and %% remove this line and below if single author
  Xiangying Huang \footnote{Department of Statistics and Operations Research, Hanes Hall, University of North Carolina, Chapel Hill, NC 27599. \EMAIL{zoehuang@unc.edu} }}
\KEYWORDS{attachment functions, continuous time branching processes, network centrality measures, degree centrality, persistence, root finding algorithms} % Separate items with ;

\AMSSUBJ{05C82, 60J85, 60J28} % Edit. Separate items with ;
\SUBMITTED{February 21, 2022} % Edit.
\ACCEPTED{March 7, 2023} % Edit.

\VOLUME{0}
\YEAR{2023}
\PAPERNUM{0}
\DOI{10.1214/YY-TN}

\ABSTRACT{We consider growing random networks $\{\GG_n\}_{n \ge 1}$ where, at each time, a new vertex attaches itself to a collection of existing vertices via a fixed number $m \ge 1$ of edges, with probability proportional to a function $f$ (called attachment function) of their degree. It was shown in \cite{BBpersistence} that such network models exhibit two distinct regimes: (i) the persistent regime, corresponding to $\sum_{i=1}^{\infty}f(i)^{-2} < \infty$, where the top $K$ maximal degree vertices fixate over time for any given $K$, and (ii) the non-persistent regime, with $\sum_{i=1}^{\infty}f(i)^{-2} = \infty$, where the identities of these vertices keep changing infinitely often over time. In this article, we develop root finding algorithms using the empirical degree structure and local network information based on a snapshot of such a network at some large time. In the persistent regime, the algorithm is purely based on degree centrality, that is, for a given error tolerance $\ep \in (0,1)$, there exists $K_{\ep} \in \mathbb{N}$ such that for any $n \ge 1$, the confidence set for the root in $\GG_n$, which contains the root with probability at least $1 - \ep$, consists of the top $K_{\ep}$ maximal degree vertices. In particular, the size of the confidence set is stable in the network size. Upper and lower bounds on $K_{\ep}$ are explicitly characterized in terms of the error tolerance $\ep$ and the attachment function $f$. In the non-persistent regime, for an appropriate choice of $r_n \rightarrow \infty$ at a rate much smaller than the diameter of the network, the neighborhood of radius $r_n$ around the maximal degree vertex is shown to contain the root with high probability, and a size estimate for this set is obtained. It is shown that, when $f(k) = k^{\alpha}, k \ge 1,$ for any $\alpha \in (0,1/2]$, this size grows at a smaller rate than any positive power of the network size.}

\begin{document}

%%%%%%%%%%%%%%%%%%%%%%%%%%%%%%%%%%%%%%%%%%%%%%%%%%%%%%%%%%%%%%%%%%%
%%                                                               %%
%% No need for \maketitle.                                       %%
%%                                                               %%
%%%%%%%%%%%%%%%%%%%%%%%%%%%%%%%%%%%%%%%%%%%%%%%%%%%%%%%%%%%%%%%%%%%

%%%%%%%%%%%%%%%%%%%%%%%%%%%%%%%%%%%%%%%%%%%%%%%%%%%%%%%%%%%%%%%%%%%
%%                                                               %%
%% Please replace what follows by the body of your article       %%
%% (up to the bibliography):                                     %%
%%                                                               %%
%%%%%%%%%%%%%%%%%%%%%%%%%%%%%%%%%%%%%%%%%%%%%%%%%%%%%%%%%%%%%%%%%%%

\section{Introduction}

The last two decades have seen an explosion in the application of temporal networks in diverse scenarios, including modeling and analyzing connectivity in groups of individuals, spread of epidemics, opinions and rumors, just to name a few. Owing to their large size, it is often infeasible to track these networks over time; instead, one observes a snapshot of the network at some large time and asks about its nascent states. In this article, we are interested in growing random networks where an incoming vertex connects to a collection of existing vertices via a fixed number of edges, with probability proportional to a function $f$ of their degree, called the attachment function. We ask the following: can one detect the root (first vertex in the network) to a quantifiable degree of accuracy based on observing a single snapshot of this network process? This question connects to the broad field of \emph{network archaeology} \cite{navlakha2011network} where one is interested in reconstructing the past of a network from its present, including connectivity structure, central nodes, etc. Although related questions have been studied in a more applied setting, a rigorous mathematical analysis has only recently been initiated (see Section \ref{cenhis}). Most mathematical works focus on the uniform and linear preferential attachment networks (when $f$ is respectively constant and linearly increasing) owing to their analytical tractability. However, from a modeling perspective, such assumptions seem rather ad hoc and ideally the attachment function should be inferred based on the network data at hand. Hence, it is important to address the above questions in a more general setting with minimal assumptions on the attachment function. In this article, we ask to what extent the information on the empirical degree structure of the network snapshot can aid in root detection. This is  part of our continuing program to understand network archaeology for growing random networks driven by general attachment functions (see also \cite{banerjee2018fluctuation,BBpersistence,jordan}).

\subsection{Model definition}\label{modeldef}
We define a slightly more general version of the class of networks studied in this paper.
Fix an attachment function $f: \mathbb{N} \rightarrow (0,\infty)$ and an $\mathbb{N}$-valued (possibly random) attachment sequence $\{m_i:i\geq 1\}$. We will construct a sequence of random graphs $\{\GG_n:n\geq 1\}$ as follows:

\mn
(i) Let $\GG_1$ be the graph with two vertices $\{v_0, v_1\}$ and $m_1$  edges between them, where $v_0$ is also referred to as the root.

\mn
(ii) Fix $n>1$. Given we have obtained $\GG_{n-1}$, $\GG_n$ is constructed from $\GG_{n-1}$ by adding one vertex $v_n$ and $m_n$ directed edges, each with one end connected to $v_n$, and the other ends of these edges are connected sequentially to one of the existing vertices $\{v_i: 0\leq i\leq n-1\}$ in $\GG_{n-1}$ according to the following rule. For $1 \leq j \leq m_n$, the $j$-th edge has its other end connected to $v_i$ $(0 \leq i \leq n-1)$ with probability proportional to $f\text{(degree of $v_i$)}$ (the degree is computed before the connection is made).

\medskip
We will find it convenient to index the graph sequence by the number of attached edges (with both ends already connected to respective vertices).
Let $s_n :=\sum_{i=1}^n m_i$ for $n\geq 1$ and $s_0=0$. Mathematically, we construct the following (directed) random graph sequence $\{ \GG_k^*: k\geq s_1\}$ where $\GG_k^*$ has $k$ attached edges. For any $n\geq 2$ and any $s_{n-1} <k\leq s_n$, $\GG_k^*$ has $n+1$ vertices. For $l \geq 0$, let $d_0(l)$ denote the degree of the root after the $l$-th edge is attached. For $ i \geq 1$, denote by $d_i(l)$, $ l > s_{i-1}$, the degree of the $(i + 1)$-th vertex (i.e. vertex $v_i$) after the $l$-th edge is attached and set $d_i(l) = 0$ for all $l \leq s_{i-1}$. 
For $s_{n-1} < k \leq s_n$, an edge $e_k$ is added to the graph, with one end attached to $v_n$, and for $0 \leq i \leq n-1$,
\begin{equation}\label{attprobab}
P(\text{other end of $e_k$ attached to $v_i$} \, | \, \GG_j^*, j\leq k-1) =\frac{f(d_i(k-1))}{\sum_{j=0}^{n-1}f(d_j(k-1))}.
\end{equation}
Thus, $\{\GG_n: n\geq 1\}$ has the same law as $\{\GG_{s_n}^* :n\geq 1\}$. We will denote the vertex set of $\GG_n$ by $V(\mathcal{G}_n)$.

Two important special cases are: (i) $f(k) = k+ \beta$ with $k \ge 1$ for some $\beta \ge 0$, which is known as the \emph{linear preferential attachment} model, and (ii) $f \equiv 1$, which is called the \emph{uniform attachment} model.

\subsection{Centrality measures, persistence and root finding algorithms}\label{cenhis}
Let $\mathbb{G}$ be a collection of (possibly labeled) graphs. We write $\mathbb{G}^{\circ}$ for the corresponding collection of unlabeled graphs and let $V(\mathbb{G}^\circ)$ denote its vertex set. A centrality measure $\Psi: V(\mathbb{G}^{\circ}) \rightarrow \mathbb{R}_+$ gives each vertex in the graph a `score' based on some notion of `centrality' in an appropriate geometry of the unlabeled graph. In our setup, $\mathbb{G} = \{\GG_n : n \ge 1\}$ is a growing sequence of random graphs and our goal is to recover the root from an unlabeled copy of $\GG_n$ for some large $n$. The idea is to choose a centrality measure $\Psi$ that \emph{correlates strongly with the age} of the vertex. Then, for a given error tolerance $\ep \in (0,1)$, we can choose the $K_\ep(n)$ most central vertices (as determined by $\Psi$) in $\GG_n$ for an appropriate $K_\ep(n) \in \mathbb{N}$ and ensure that the root of $\GG_n$ is contained in this set with probability at least $1-\ep$.

There are several existing notions of centrality, some of which we now discuss.
\begin{itemize}
\item[(i)] \emph{Degree centrality} \cite{SPdegree,galashin2013existence,BBpersistence} quantifies the centrality of a vertex in a general network simply based on how large its degree is. For dynamic networks, the older vertices exist in the network for a longer time and one expects them to have a higher degree. 
\item[(ii)] \emph{Jordan (or centroid) centrality} \cite{bubeck2017findingup,jog2018persistence,jog2016analysis,jordan} quantifies centrality in trees by looking at the size of the maximal subtree rooted at each vertex. The smaller this size, the more comparable the sizes are of all the subtrees rooted at this vertex, and the more central this vertex is.
\item[(iii)] \emph{Rumor Centrality} \cite{shah2011rumors,shah2012rumor,khim2016confidence} was originally developed in connection to finding the most likely source of rumor spread or infection spread for a susceptible-infected (SI) epidemic model on a background random graph. This measure has been successfully used in devising root finding algorithms for the uniform attachment model in \cite{bubeck2017findingup}.
\end{itemize}
There are several other notions of centrality \cite{freeman1977set,newman2005measure,boldi2014axioms}, but only the ones above have so far been rigorously mathematically analyzed in connection with root detection. Recently, \cite{bubeck2015influence,curien2015scaling} and \cite{bubeck2017trees,lugosi2019finding,devroye2018discovery} have respectively used degree and Jordan centrality in reconstruction of seeds (initial graph from which the network starts) from a single snapshot of the network at a large time in linear preferential attachment and uniform attachment models.

A very desirable property of centrality measures is \emph{persistence} as defined below. 

\begin{definition}
 Fix $K \geq 1$ and a network centrality measure $\Psi$. For an evolving network sequence $\{\GG_n : n \geq  1\}$, say that the sequence is $(\Psi, K)$ persistent if there exists $n^*<\infty$ a.s. such that
for all $n \geq n^*$ the optimal $K$ vertices $(v^{(n)}_{1,\Psi},v^{(n)}_{2,\Psi} ,\dots, v^{(n)}_{K,\Psi} )$, as measured by the centrality measure $\Psi$, remain the same and further the relative ordering amongst these $K$ optimal vertices remain the same.
\end{definition} 

Persistence ensures that the identities of the most central vertices fixate over time. This implies that one can use the centrality measure $\Psi$ to construct confidence sets for the root, with given error tolerance $\ep$, whose size only depend on $\ep$ and not on the network size, that is, $K_{\ep}(n) \equiv K_{\ep}$. This stability with respect to network size makes the confidence set robust for large networks.

Persistence of the degree centrality measure was first shown for the linear preferential attachment model in \cite{galashin2013existence}.  The phenomenon of persistence for the general class of models described in Section \ref{modeldef} was first investigated in \cite{BBpersistence}. Building on the techniques of \cite{SPdegree} (which investigated persistence for a related dynamic random graph model), they showed the existence of two distinct regimes: (i) the \emph{persistent regime} when $\sum_{i=1}^{\infty}\frac{1}{f(i)^2}<\infty$, and (ii) the \emph{non-persistent regime} when $\sum_{i=1}^{\infty}\frac{1}{f(i)^2} = \infty$. We now summarize some of the results of \cite{BBpersistence}. In the following, for any $v \in \GG_n$, $\Psi(v)$ is the degree of $v$ in $\GG_n$. Moreover, $m_i$ can be random for the results stated below.

\begin{itemize}
\item[(i)] \emph{Persistent regime} : Assume $f$ is non-decreasing and there exists $C_f>0$ such that $f(i)\leq C_f i$ for all $i\geq m_1$. Also, suppose $\sum_{i=1}^{\infty}\frac{1}{f(i)^2}<\infty$ and that, almost surely,
$$\limsup_{n\to\infty} \frac{1}{\log s_n} \sum_{i=1}^{m_n-1}\frac{1}{f(i)} \leq \frac{1}{8C_f}.$$
Then, the random graph sequence $\{\GG_n\}_{n\geq 1}$ is $(\Psi,K)$ persistent for any fixed $K \ge 1$. In case $m_i=1$ for all $i \ge 1$ (the tree case), the monotonicity and sub-linearity assumptions on $f$ can be replaced by Assumption (A2) below.

\item[(ii)] \emph{Non-persistent regime} : If $\sum_{i=1}^{\infty}\frac{1}{f(i)^2} = \infty$ and a `continuity' assumption \eqref{contK} on the attachment function is satisfied, then, almost surely, the random graph sequence $\{\GG_n\}_{n\geq 1}$ is not $(\Psi,K)$ persistent for any $K \ge 1$.
\end{itemize}

The above suggests that, in the persistent regime, there is hope for devising degree centrality based root finding algorithms which produce confidence sets that are stable in the network size. In the non-persistent regime, the stability of purely degree based algorithms in the network size seems unlikely. It is then natural to ask how much of the network connectivity information is required to obtain efficient root finding algorithms.

In \cite{jordan}, it was shown that the Jordan centrality measure (defined only for the tree case $m_i \equiv 1$) exhibits persistence more generally whenever the attachment function $f$ satisfies Assumptions (A1)-(A3) below. This observation was used to obtain explicit confidence set bounds for root finding algorithms that grow polynomially in $\ep^{-1}$ ($\ep$ being the error tolerance). However, this centrality measure is applicable only for tree networks. Moreover, it is computationally expensive as computing the Jordan centrality measure for any single vertex requires knowledge of the connectivity structure of the whole network.

This leads to some natural questions. (i) Can we devise root finding algorithms for the dynamic networks defined in Section \ref{modeldef} for $m_i \equiv m \ge 1$ in the persistent regime ($\sum_{i=1}^{\infty}\frac{1}{f(i)^2}<\infty$), purely based on the empirical degree distribution, that are stable in the network size? In this case, can we give explicit bounds for the confidence set? (ii) In the non-persistent regime, can we obtain root finding algorithms using degree centrality and minimal local connectivity information around high degree vertices? What is the size of the associated confidence set and how does it grow with network size? These questions, besides their theoretical appeal, are extremely important from an applications perspective as most real world networks are large and far from tree like. Degree based centrality measures are much easier to compute than their global analogues (like Jordan centrality) as they are purely local. Moreover, they are well-defined and applicable on non-tree networks.

\subsection{Our contribution}
We attempt to answer the above questions in the current article.
In the process, we develop a quantitative approach to degree centrality. Theorem \ref{linearattach} obtains sharp upper and lower bounds on the size of the confidence set for the root (or budget) in terms of the inverse error tolerance for the linear preferential attachment model for general $m \ge 1$. Theorem \ref{general} addresses the general persistent regime. For trees, the bounds are almost optimal. Even for $m >1$, the bounds are polynomial in the inverse error tolerance and are explicitly obtained in terms of the attachment function. It is noted in Remark \ref{pereff} that, in the persistent regime, not only does one obtain purely local root finding algorithms that do not require a network tree structure as in \cite{jordan}, but also the bounds obtained are nearly as good as, and in most cases tighter than, those for Jordan centrality. A surprising observation, noted in the same remark, is that, as seen explicitly through the size of the confidence set in the linear attachment case in Theorem \ref{linearattach} , the efficacy of root detection using degree centrality gets better with increasing $m$. Thus, although the network becomes increasingly non-tree like for large $m$, the increased positive reinforcement of degrees allows the older vertices to gain in degree even further.

Next, we analyze the non-persistent regime. In Theorem \ref{nonpersistent}, we show that a neighborhood of the maximal degree vertex of an appropriately chosen radius $r_n \rightarrow \infty$ contains the root with high probability. For given error tolerance $\ep$, we also quantify the minimum network size for the above set to contain the root with probability at least $1-\ep$. We note in Remark \ref{nprem} that for large $n$ the radius $r_n$ is much smaller than the diameter of the network, and for sublinear attachment functions $f(k) = k^{\alpha}, k \ge 1,$ for any $\alpha \in (0,1/2]$, the size of the confidence set grows at a smaller rate than any positive power of the network size. In Remark \ref{da} we list some results (which essentially follow from \cite{BBpersistence}) on the degree distribution of fixed vertices and the maximum degree as the network grows. In particular, the degrees of fixed vertices scale similarly as the maximum degree in the persistent regime, but show a stark contrast in the non-persistent regime.

Although we use a number of ideas from \cite{BBpersistence} and \cite{jordan}, there are several significant new techniques and ideas developed here, some of which we now describe. We introduce a novel embedding of the network in the case $m_i \equiv m >1$ into a collapsed continuous time branching process (see Section \ref{BPpre}) which lets us bypass the monotonicity assumptions required on $f$ in \cite{BBpersistence} in the non-tree case. We also develop a unified approach to budget lower bounds (see Section \ref{lowerunif}) and apply it in Lemmas \ref{lbpart1}, \ref{lbm1} and \ref{lbgm}. A major technical advance in the current article is the detailed quantification of the asymptotic behavior of degrees, population sizes of associated continuous time branching processes (CTBP), the age of the maximal degree vertex and its distance from the root. To do this, we exploit properties of recursive distributional equations which characterize limiting normalized population sizes for CTBP (see Lemma \ref{lefttail}) extending the program of \cite{jordan}. Moreover, the proof of Theorem \ref{nonpersistent} relies on novel connections with asymptotics of the height of CTBP (maximal distance of root to any vertex), and characterizing the behavior of functionals of associated Laplace transforms near zero (see Lemmas \ref{vmaxerr} and \ref{alphax}).

We note here that for linear preferential attachment networks, there is an extensive body of literature (see \cite{bollobas2001degree,pekoz2017joint,senizergues2021geometry} and references therein) obtaining detailed asymptotics for the degree evolution of vertices, the limits of (appropriately normalized) degrees to explicit random variables and rates of convergence to these limits. It would be interesting to obtain bounds on the confidence sets by appealing to these results in the linear case and compare them to the ones obtained in the current article. However, our goal here is to develop general techniques that are robust enough to apply to a much wider class of attachment functions, and we stick to this principle throughout the paper.

Section \ref{mainres} states the main results and an outline of proof techniques. Section \ref{BPpre} describes branching process embeddings for the network and collects several tail estimates for important quantities of this branching process. Section \ref{mardeg} collects properties of some key martingales used to track degrees of individual vertices. Section \ref{lowerunif} describes a universal approach for obtaining budget lower bounds by showing that, if the root has relatively small degree in the graph $\GG_K$, then with probability at least $1/2$, its centrality score cannot beat those of an $\Omega(K)$ number of vertices at any subsequent time. Probabilities of such configurations in $\GG_K$ are then estimated in respective cases. Section \ref{linsec}, \ref{gensec}, \ref{nonsec} are respectively devoted to the proofs of Theorems \ref{linearattach}, \ref{general} and \ref{nonpersistent}.

We note here that, for linear preferential attachment models, there is an alternate notion of progressive exploration towards finding the root in large networks where one performs local moves towards vertices of higher degree \cite{borgs2012power,frieze2017looking,brautbar2010local}. However, this leads to a confidence set that grows polylogarithmically in $n$ (as opposed to our confidence set, whose size is stable in $n$), with the computational advantage of only having to explore a small part of the graph (compared to our method requiring the degree of each vertex in the network). We plan to investigate such exploratory root finding algorithms for more general attachment networks in a future article.

\mn
\textbf{Notation: }Throughout this paper, $C,C',C''$ will be used to denote generic positive constants whose values might change from line to line. Unless otherwise stated, all constants are implicitly assumed to depend on $m$ (number of edges attached per new vertex). We also set the convention `$\sum_{i=a}^b = 0$' whenever $a >b$. For two c\`adl\`ag processes $\chi_1, \chi_2$, we will write $\chi_1(\cdot) \overset{d}{\geq} \chi_2(\cdot)$ if there is a coupling $(\tilde{\chi}_1(\cdot),\tilde{\chi}_2(\cdot))$ such that $\tilde{\chi}_1(t) \geq \tilde{\chi}_2(t)$ for all $t \ge 0$.

\medskip
We now define some key objects used in this article. Let $\{E_i: i\geq 1\}$ be i.i.d. exponential random variables with mean 1. For $k \geq 1$, define
\begin{equation}\label{sdef}
S_k(l):=\sum_{i=1}^{l}\frac{E_i}{f^k(i)}, \quad l \in\mathbb{N}.
\end{equation}
%Here and throughout the rest of the article, we stick to the convention $\sum_{j=0}^{-1} = \sum_{j=1}^{0} = 0$.
Define the point process $\xi$ by 
\beq\label{pprocess}
\xi(t):=\sum_{n=1}^\infty 1_{\{S_1(n)\leq t\}}, \quad t\geq 0.
\eeq
We will write $\sigma_l := S_1(l), \, l \ge 1,$ to denote the arrival time of the $l$-th individual in the above point process.
For $ A \in\mathbb{N}_0$, denote by $\xi_A(\cdot)$ the point process obtained as in \eqref{pprocess} with attachment function $f_A(\cdot) := f(A+\cdot)$ replacing $f(\cdot)$. Note that $\xi_0(\cdot) \overset{d}{=} \xi(\cdot)$. By convention, for any $A\in \mathbb{ N}_0$, we take $\xi_A(t) = 0$ for $t < 0$.

Define the functions
\beq\label{phd}
\Phi_k(l):=\sum_{i=1}^{l-1}\frac{1}{f^k(i)}, \quad l,k\in\mathbb{N}.
\eeq
Extend $f$ to a function on $[1, \infty)$ by requiring $f(x) = f(\lfloor x \rfloor)$ where $x \ge 1$. Extend $\Phi_k(\cdot), k = 1, 2,$ to $[1,\infty)$ via linear interpolation and set $\Phi_k(z)=0$ for $z \in [0,1]$. Note that, using the above extensions of
$f, \Phi_1,\Phi_2$,
$$\Phi_k(x)=\int_1^x \frac{1}{f^k(z)}dz, \quad\quad x\geq 1, k=1,2.$$
 Define
 $$\KK(t) :=\Phi_2\circ \Phi_1^{-1}(t),\quad\quad t\geq 0.$$
We will use the notation $\Phi_k(\infty):=\lim_{l\to\infty}\Phi_k(l)$.

\subsection{Assumptions on attachment functions}
Let $\mu(t) := E(\xi(t)),t \geq 0$. We can naturally obtain a non-negative measure on $(\RR_+,\mathcal{B}(\RR_+))$ from $\mu$ (the intensity measure of $\xi$) which we also denote by $\mu$. We define the Laplace transform of $\mu$  by
$$\hat{\mu}(\theta)=\int_0^\infty e^{-\theta t}\mu(dt), \quad  0 \le \theta<\infty.$$
It is known that (see \cite[Equation (3)]{rudas2007random})
$$\hat{\mu}(\theta)=\sum_{k=1}^\infty \prod_{i=1}^k \frac{f(i)}{\theta+f(i)}.$$
Throughout this article, we have the following assumptions on $f$:\\\\
%(A1) $f_*\equiv \min_{i\geq m_1} f(i) >0$.
(A1) Every attachment function satisfies $f(k) \rightarrow \infty$ as $k \rightarrow \infty$, and $f$ can grow at most linearly, i.e., there exists $C_f>0$ such that $f(i)\leq C_f i$ for all $i\geq m_1$.

\mn
(A2) Define $\underline{\lambda}:=\inf\{ \lambda>0: \hat{\mu}(\lambda)<\infty\}.$ We assume,
\begin{equation*}
\underline{\lambda}<\infty \quad \text{and} \quad \lim_{\lambda \downarrow \underline{\lambda}} \hat{\mu}(\lambda)>1.
\end{equation*}

Assumption (A2) together with the monotonicity of $\hat{\mu}(\cdot)$ implies that there exists a unique $\lambda^*:=\lambda^*(f)$ such that 
\beq\label{malthusian}
\hat{\mu}(\lambda^*)=1.
\eeq
This value $\lambda^*$ is often referred to as the \emph{Malthusian rate} of growth parameter. The above assumptions are relatively standard in the continuous time branching process literature, and will be essential for continuous time embeddings of our discrete network process discussed in Section \ref{BPpre}.

\mn
(A3) With $\lambda^*$ denoting the Malthusian rate of growth parameter, $\limsup_{i\to\infty}f(i)/i<\lambda^*$.\\

Observe that Assumption (A1) implies that $f_* := \min_{i\geq m_1} f(i) >0$.

\begin{remark}
Suppose $f$ satisfies Assumptions (A1) and (A2). Then, by \cite[Lemma 2.3]{jordan}, Assumption(A3) is satisfied if any of the following conditions hold: (a) there exists $C^*\geq 0$ such that $\lim_{i\to\infty} f(i)/i = C^*$; (b) $\limsup_{i\to\infty} f(i)/i < f_*$.  
	In particular (affine) preferential attachment and sublinear preferential attachment (there exists $\alpha \in (0,1)$ such that $f(k) \le k^{\alpha}$ for all $k \ge 1$) satisfy Assumptions (A1)-(A3). 
\end{remark}

\section{Main results}\label{mainres}

Our main results give bounds on the minimum number of central vertices, measured by some centrality measure, required to produce a confidence set of the root for a given error tolerance. This is quantified through the notion of \emph{budget} as defined below.

Let $\mathbb{G}$ be a collection of (possibly labeled) graphs, and let $\mathbb{G}^\circ$ denote the corresponding set of unlabeled graphs.
For any $K\geq 1$ and a centrality measure $\Psi$, we will write $H_{K,\Psi}: \mathbb{G} \rightarrow V(\mathbb{G})$ to denote the function that maps $G\in \mathbb{G}$ to a subset of $K$ most central vertices in $G$ as measured by $\Psi$ (ties being broken arbitrarily).
\begin{definition}\label{budget}
Fix $0 < \ep < 1$ and $K \geq 1$. A mapping $H_{K,\Psi}$ is called a budget $K$ root finding algorithm with error tolerance $\ep$ for the collection of random graphs $\{\GG_n : n \geq 1\}$ if
\begin{equation}\label{bug}
\liminf_{n\to\infty} P(v_0 \in H_{K,\Psi}(\GG_n)) \geq 1- \ep.
\end{equation}
We will write $K_{\Psi}(\ep)$ to denote the smallest budget $K$ such that $H_{K,\Psi}(\cdot)$ is a root finding algorithm with error tolerance $\ep$ for the collection of random graphs $\{\GG_n: n\geq 1\}$. If there is no finite $K$ satisfying \eqref{bug}, we set $K_{\Psi}(\ep)= \infty$.
\end{definition}

\subsection{The persistent regime}

In this section, we restrict attention to $\Psi$ being the degree centrality measure, that is $\Psi(v) = deg(v), \ v \in V(\mathbb{G}^\circ)$. We also focus on the persistent regime, namely, attachment functions satisfying $\Phi_2(\infty) < \infty$. Our first result gives sharp upper and lower bounds on $K_{\Psi}(\ep)$ when $f$ is linear.

\begin{theorem}[Linear attachment]\label{linearattach}
Suppose $f(i)= i+\beta$ for $i\geq 1,$ for some $\beta\geq 0$, and $m_i \equiv m \ge 1$. There exist positive constants $C_1,C'_1,C_2$, depending on $m,\beta$, such that for any $\ep \in (0,1)$,
$$ \frac{C'_1}{\ep^{\frac{2m+\beta}{m(m+\beta)}}}\leq K_\Psi(\ep)\leq \frac{C_1}{\ep^{\frac{2m+\beta}{m(m+\beta)}}}\exp\left(\sqrt{C_2\log \frac{1}{\ep}}\right).$$

\end{theorem}

The next result considers the persistent regime in generality. Part (i) gives near optimal bounds for $K_{\Psi}(\ep)$ in the tree case in terms of the Malthusian rate $\lambda^*$ and  $f_* := \min_{i\geq m_1} f(i)$. Part (ii) gives upper and lower bounds for $K_{\Psi}(\ep)$ in the non-tree case that are polynomial in $\ep^{-1}$.
\begin{theorem}\label{general}
Suppose the attachment function $f$ satisfies Assumptions (A1)-(A3). Moreover, assume $\Phi_2(\infty)<\infty$.

\mn
(i) Suppose $m_i\equiv m= 1$. 
For any fixed $\delta \in (0,1)$, there exist positive constants $C_1$ (not depending on $\delta$) and $C_{\delta}$ (depending on $\delta$) such that for all $\ep \in (0,1)$,
$$\frac{C_1}{\ep^{\frac{\lambda^*}{f_*}}}\leq K_\Psi(\ep)\leq \frac{C_\delta}{\ep^{\frac{\lambda^*}{(1-\delta)f_*}}}.
%\exp\left(\sqrt{C_2\log \frac{1}{\ep}}\right),
$$
%where $C_1$ is a universal positive constant and $C_\delta$ is a positive constant depending on $\delta$.

\mn
(ii) Suppose $m_i\equiv m>1$. For any $\delta \in (0,1)$, there exist positive constants $C_1$ (not depending on $\delta$) and  $\bar{C}_{\delta}$ (depending on $\delta$) such that for all $\ep \in (0,1)$,
$$\frac{C_1}{\ep^{\frac{f_*}{mf(m)}}}\leq K_\Psi(\ep) \leq \frac{\bar{C}_{\delta}}{\ep^{\frac{2C_f}{(1-\delta)f_*}}}, \ \ep \in (0,1).$$
%for some $c_f>\frac{1}{m}$, where $C_1'$ is a positive constant.

\end{theorem}

\begin{remark}\label{pereff}
Several remarks are in order.
\begin{itemize}
\item[(a)] Note that if $f(k) \le g(k)$ for all $k \ge 1$, then $\lambda^*(f) \le \lambda^*(g)$. Using this and the known Malthusian rates for uniform attachment and linear preferential attachment, we conclude $f_* \le \lambda^* \le 2C_f$. From this observation, it follows that the bounds in (i) are indeed tighter than those in (ii) (applied to $m=1$). 
\item[(b)] In \cite{jordan}, polynomial upper and lower bounds in $\ep^{-1}$ are derived for $K_{\Psi}(\ep)$ when $\Psi$ is the Jordan centrality measure. However, this measure is only applicable in the tree case and is computationally expensive as the score for each vertex requires information on the size of all the subtrees rooted at that vertex. The degree centrality measure works on non-tree networks and its computation only requires the degree of each vertex. Moreover, the exponent of $\ep^{-1}$ in the upper bound on $K_{\Psi}(\ep)$ for the Jordan centrality measure obtained in \cite[Theorem 3.1 (a)]{jordan} is $2C_f/f_*$ (taking $\beta=0$ there), whereas the corresponding exponent in (i) above is $\lambda^*/(1-\delta)f_*$ (for arbitrarily small $\delta>0$). As $\lambda^*/f_* \le 2C_f/f_*$ (see (a)), we conclude that, in the tree case under the persistent regime, the obtained upper bound on the size of the confidence set for the root using degree centrality is comparable to, and in most cases tighter than, that in \cite{jordan} using Jordan centrality. 

\item[(c)] In the tree case, when $f$ is linear, from Theorem \ref{linearattach} with $m=1$ and \cite[Corollary 3.3 (c)]{jordan}, the exact exponent of $\ep^{-1}$ in the order of $K_{\Psi}(\ep)$ in both degree and Jordan centrality measures is seen to be $(2+\beta)/(1+\beta)$. In this sense, degree centrality does just as good as Jordan centrality on trees for linear $f$, although it is computationally far more efficient.

\item[(d)] The exponent of $\ep^{-1}$ in the upper bound in Theorem \ref{general} (ii) equals that in Theorem \ref{linearattach} (taking $\delta \sim 0$) in the linear case with $\beta=0$. As Theorem \ref{linearattach} captures the exact exponent, the upper bound in Theorem \ref{general} (ii), in this sense, is the best one can obtain. In this case, the lower bound exponent is clearly not tight, and it is unclear if there are examples of $f$ for which the lower bound is tight in the non-tree case. The technical difficulty in getting sharp bounds in this set-up is that the underlying network does not admit an embedding into a `true'  continuous time branching process (see Section \ref{BPpre}) which renders many of the tools we use to obtain refined asymptotics in the tree case inapplicable. 

\item[(e)] The exponent of $\varepsilon^{-1}$ in the bounds obtained in Theorem \ref{linearattach} is seen to decrease as $m$ grows. This surprising observation implies that, for large $m$, although the network becomes increasingly non-tree like, the efficacy of root detection using degree centrality increases. This can be intuitively explained by the stronger positive reinforcement of degrees as $m$ grows: each new vertex comes in with more edges which tend to increase the degrees of high degree vertices even more.

\item[(f)] The constants appearing in the above theorems implicitly depend on $m$. It will be interesting to make this dependence more explicit. It will also be interesting to extend the above results to non-constant and possibly random attachment sequences $\{m_i\}$. We will address this in subsequent work.
\end{itemize}
\end{remark}

\subsection{The non-persistent regime}

Consider the non-persistent regime where $\Phi_2(\infty)=\infty$ and suppose $f(k)\to\infty$ as $k\to\infty$. In this regime, under some mild assumptions on $f$ (see 1, 2 in Theorem \ref{nonpersistent}), the identity of the maximal degree vertex changes infinitely often as the network grows \cite{BBpersistence}. Moreover, for any fixed collection of vertices, the probability of the maximal degree vertex lying in this collection goes to zero as the network grows (see \eqref{indas}). Thus, there is no hope of obtaining a confidence set for the root solely based on empirical degree information that is stable in the network size. On the other hand, \cite{jordan} showed that a stable confidence set can be obtained for tree networks, but this requires connectivity information for the whole network via Jordan centrality and hence is computationally expensive on large networks. In this section, we exhibit how the connectivity information in a neighborhood around the maximal degree vertex of radius much smaller than the network diameter can be used to furnish confidence sets of size much smaller than the network size.

In the tree case, \cite[Theorem 4.12]{BBpersistence} sheds light on the asymptotics of the age of the maximal degree vertex as a function of network size. More precisely, define $\mathcal{I}_n^* := \inf\{1 \le i \le n: d_i(n) \ge d_j(n) \, \forall \, j \le n\}$ for $n \ge 1$. Then, \cite[Theorem 4.12]{BBpersistence} shows that
\begin{equation}\label{indas}
\frac{\log \mathcal{I}_n^*}{\KK(\frac{1}{\lambda^*}\log n)} \xrightarrow{P} \frac{(\lambda^*)^2}{2}, \ \ \text{as} \ \ n \rightarrow \infty.
\end{equation}
Thus, if a comparison is obtained between the age of the maximal degree vertex and its distance from the root, that is, $\rho_n := \operatorname{dist}(v_{\mathcal{I}^*_n}, v_1)$, then the above result can be used to obtain the asymptotics of $\rho_n$. Then, one can expect to `find the root' inside a ball around the maximal degree vertex of radius comparable to $\rho_n$. This is the idea behind the following theorem. In Theorem \ref{nonpersistent}, $r_n$ can be thought of as $\rho_n$ and $b_n$ represents the approximate number of vertices in the ball of radius $r_n$ around the maximal degree vertex.

For a sequence $\{r_n\}_{n \ge 1} \subset \mathbb{N}$, we define $B_n(v,r_n)$ to be the set of all vertices in $\GG_n$ which are at graph distance at most $r_n$ from $v$. Let $|B_n(v,r_n)|$ denote the number of vertices in $B_n(v,r_n)$.
\begin{theorem}\label{nonpersistent}
Suppose $m_i\equiv m= 1$. Suppose $f(i)\leq C_0i^{\alpha}$ for $i \ge 1,$ for some $C_0>0$ and $\alpha\in (0,1/2]$. Moreover, assume that $f(k)\to\infty$ as $k\to\infty$. Also assume that $\KK(\cdot) := \Phi_2 \circ \Phi_1^{-1}(\cdot)$ satisfies the following `continuity' assumptions:
\begin{itemize}
\item[1.]
$\lim_{\delta \downarrow 0}\limsup_{t\to\infty} \frac{\KK((1+\delta)t)}{\KK(t)}=1.$
\item[2.] There exists positive constants $t'$, $D$ such that $\KK(3t)\leq D\KK(t)$ for all $t\geq t'$.
\end{itemize}
Define
\begin{align*}
r_n&:=c_1\lambda^*\KK(\frac{1}{\lambda^*}\log n),\\
b_n&:=\exp\left( \frac{8}{1-\alpha}r_n \log\left( \frac{\log n}{2\lambda^*r_n}\right) \right),
\end{align*}
where $c_1$ is a positive constant (independent of $n$) defined in Lemma \ref{height}.
%Let $H_{b_n,\Psi_{n,r_n}}(\GG_n)$ denote the set of $b_n$ vertices in $\GG_n$ with the largest $\Psi_{n,r_n}$ values (ties being broken arbitrarily).
There exist $C,C'>0$ (independent of $n$) such that for any $\ep \in (0,1)$, 
$$P\left(v_0 \in B_n(v_{max}(n),r_n), \,|B_n(v_{max}(n),r_n)| \le b_n\right) \geq 1-\ep \ \text{ for all } \ n\geq \exp\left( \lambda^* \KK^{-1}\left(\frac{\log(C/\ep)}{C'}\right)\right).$$
\end{theorem}

In words, the above theorem says that, with high probability, the root lies within a neighborhood of radius $r_n$ around the maximal degree vertex and the size of this (confidence) set is at most $b_n$. For arbitrary $\ep>0$, it also gives a lower bound on the network size for the root to lie in this set with probability at least $1-\ep$. This lower bound is especially useful in the non-persistent regime as the size of the confidence set grows with network size and this makes the result applicable for fixed large $n$.

\begin{remark}\label{nprem}
We make the following comments.

(a) To compute the above confidence set for the root, one only needs to know the maximal degree vertex (which can be obtained from the empirical degree distribution) and the network in a neighborhood of radius $r_n$ around this vertex. Under the hypotheses of Theorem \ref{nonpersistent}, $r_n/\log n \rightarrow 0$ as $n \rightarrow \infty$, by \eqref{ktozero}. As the diameter of $\GG_n$ (same order as the height) is typically $O(\log n)$ \cite[Theorem 13.2]{HJarxiv}, the radius of this neighborhood is much smaller than the network diameter.\\

(b) For any $\alpha \in (0, 1/2)$, taking $f(k) = k^{\alpha}$ with $k \ge 1,$ it is easy to verify that $\KK(t) = O(t^{(1-2\alpha)/(1-\alpha)})$. Thus, $r_n = O((\log n)^{(1-2\alpha)/(1-\alpha)})$. Moreover, $b_n$, which can be seen as the size of the confidence set which contains the root with high probability, is of order $$\exp\left(O\left((\log n)^{(1-2\alpha)/(1-\alpha)} \log \log n\right)\right),$$ which grows smaller than any positive power of $n$. 
 
 For $\alpha = 1/2$, $\KK(t)=O(\log t)$ and thus $r_n = O(\log\log n)$, $b_n = \exp\left(O\left((\log n) \log \log n\right)\right)$. Further, as $\alpha$ approaches $0$, $r_n$ gets closer to the diameter $\log n$ whereas, for any $\alpha > 1/2$, $r_n = O(1)$ (note that $r_n$ is well-defined for any $\alpha \in [0,1]$). Thus, Theorem \ref{nonpersistent} `interpolates' between the persistent regime and the uniform attachment model (which can be seen as the `worst case' for root detection using empirical degree distribution and local information).\\

(c) A natural question that arises is whether an appropriately constructed `purely local' centrality measure $\Psi$ (the $\Psi$-value of each vertex requires knowledge of the network only in an $O(1)$ radius) can be shown to be persistent in the (degree) non-persistent regime. This, in turn, would lead to a stable budget $K_{\Psi}(\ep)$ that does not depend on network size.  
 
Preliminary calculations show that natural guesses like the `$j$-hop degree centrality measure', defined as the sum of degrees of all vertices in a neighborhood of radius $j$ around a vertex, will not work for any fixed $j \ge 1$. To see this in the setup of (b) above for $j=1$, one can study the evolution of the sum of degrees in a $1$-neighborhood of a vertex through the continuous time processes
 \small
 $$
\xi^{(2)}_A(t) = 2 \xi_A(t) + \sum_{i=1}^{ \xi_A(t)}\xi_{0, i}(t - \sigma_{A,i}), \ t \ge 0, \ A \in \mathbb{N}_0,
 $$ 
 \normalsize
where $\xi_A$ is defined after \eqref{pprocess}, $\sigma_{A,i}$ are the event times in $\xi_A$, and $\xi_{0, i}$ are i.i.d. copies of $\xi$ defined in \eqref{pprocess}. Suppose at some time $T$, a new vertex $v$ enters the network when the root has a large degree $A$. As $\Omega(A)$ of its neighbors have comparable $\Omega(A)$ degrees, with high probability, the sum of degrees in the $1$-neighborhood around the root at time $T$ is $\ge \eta A^2$, for some small $\eta>0$. Then, $1+\xi^{(2)}_0(t)$ and $\eta A^{2} + \xi^{(2)}_A(t)$ respectively denote the sum of degrees in the $1$-neighborhood of $v$ and a lower bound on the corresponding quantity for the root, at time $t+T$. The techniques in \cite{BBpersistence} can be used to observe that persistence requires, for sufficiently large $A$, with high probability, $\eta A^{2} + \xi^{(2)}_A(t) > 1 + \xi^{(2)}_0(t)$ for all $t \ge 0$, for sufficiently small $\eta>0$.
It can be shown (using exact computation or the martingales defined in Section \ref{mardeg}) that, for any $A \in \mathbb{N}$, $E\left(\xi^{(2)}_A(t) - \xi^{(2)}_0(t)\right) = O(t^{(1+\alpha)/(1-\alpha)})$ whereas $Var\left(\xi^{(2)}_A(t) - \xi^{(2)}_0(t)\right) \ge Ct^{3/(1-\alpha)}$ for large $t$, for some $C>0$. This indicates that, for large $t$, the fluctuations of $\eta A^{2} + \xi^{(2)}_A(\cdot) -1 - \xi^{(2)}_0(\cdot)$ are of a higher order than its drift, and hence the associated sum of degrees processes cross each other infinitely often, implying lack of persistence.

 We will investigate this question in broader generality in a future article.
\end{remark}

\begin{remark}[Degree evolution of fixed vertices and maximum degree asymptotics]\label{da}
The approach of \cite{BBpersistence} and the current article can, in addition, be used to furnish, for any attachment function satisfying Assumptions (A1)-(A3), asymptotics of the degrees of the first $k$ vertices of the network for any finite $k$, and the maximum degree, as the network grows. In particular, it follows from \cite{BBpersistence} (see Theorem 4.9 and equation (8.23) there) that in the persistent regime ($\Phi_2(\infty)< \infty$), with $m_i\equiv 1$, the degree $d_k(n)$ of the $k$-th vertex when the network has $n \ge k$ edges satisfies
$$
d_k(n) = \Phi_1^{-1}\left(\frac{1}{\lambda^*}\log n + X_k(n)\right)
$$
where $X_k(n)$ converges almost surely to some finite random variable $X_k(\infty)$ as $n \rightarrow \infty$.  A similar statement holds for the maximum degree $d_{max}(n)$ as described in \cite[Theorem 4.9]{BBpersistence}. For linear $f$, the non-tree network, and asymptotics on it, can be obtained from the tree network by a collapsing procedure \cite{garavaglia2018trees,banerjee2023local}. In this case, $\Phi_1$ is logarithmic and this implies that $n^{-1/\lambda^*}d_k(n)$ and $n^{-1/\lambda^*}d_{max}(n)$ converge to non-degenerate random variables, as was noted in detail in several previous works (eg. \cite{pekoz2017joint,senizergues2021geometry}). With some more work, it should be possible to extend the above convergence to joint convergence for any finite collection of vertices, for any attachment function satisfying Assumptions (A1)-(A3).  For general $f$, characterizing the joint distribution of $\{X_k(\infty) : k \in \mathbb{N}_0\}$, the limiting rescaled maximum degree, and obtaining results in the non-tree case, are interesting open questions.

In the tree case for the non-persistent regime ($\Phi_2(\infty)= \infty$), when $\mathcal{K}(\cdot)$ satisfies a `slow-variation' type assumption, namely $\lim_{\delta \downarrow 0}\limsup_{t \rightarrow \infty} \frac{\mathcal{K}((1+\delta)t)}{\mathcal{K}(t)}=1$, there is a stark difference between the degree growth of fixed vertices and the maximum degree as the network size increases. From \cite[Lemma 7.9]{BBpersistence}, and Lemma \ref{ctbpem} and \eqref{tnlim} below, it can be shown that for any fixed $k$,
$$
\frac{\Phi_1(d_k(n)) - \frac{1}{\lambda^*}\log n}{\sqrt{\mathcal{K}\left(\frac{1}{\lambda^*}\log n\right)}} \Rightarrow N(0,1) \ \ \text{as } \ n \rightarrow \infty.
$$
However, from \cite[Theorem 4.12]{BBpersistence},
$$
\frac{\Phi_1(d_{max}(n)) - \frac{1}{\lambda^*}\log n}{\mathcal{K}\left(\frac{1}{\lambda^*}\log n\right)} \xrightarrow{P} \frac{\lambda^*}{2} \ \ \text{as } \ n \rightarrow \infty.
$$
We hope to explore these questions in more detail in future work.
\end{remark}

\subsection{Outline of proof techniques} We end this section with an outline of techniques used to prove our main results.
\begin{enumerate}
\item In Section \ref{BPpre}, we discuss embeddings of our discrete time network process into continuous time branching processes where individuals reproduce according to point processes with inter-arrival times having means determined by the attachment function. Among other things, a key technical advantage of such embeddings is that, informally speaking, they allow for a `time-changed' interpretation of the process which enables us to bypass the random denominator appearing in the network attachment probabilities \eqref{attprobab}. In the tree case, this embedding is standard and goes by the name of Athreya--Karlin embedding. For the non-tree case, we introduce a novel embedding into a collapsed branching process (CBP) formed by collapsing batches of individuals into single vertices, appropriately preserving the connectivity structure. Such CBPs have appeared recently \cite{garavaglia2018trees,banerjee2023local}, but our CBP differs from them and allows for an exact embedding of the network for all attachment functions (as opposed to just the linear and uniform case in \cite{garavaglia2018trees,banerjee2023local}). Moreover, this embedding lets us bypass the monotonicity assumptions on $f$ in the non-tree case required in several prior works (eg. \cite{BBpersistence}).

For these branching processes, we discuss refined quantitative estimates on hitting times and distributional properties of limits of renormalized population sizes (Section \ref{ctbpas}), asymptotics of the age of maximal degree vertices (Section \ref{agesec}) and height of the skeleton of the branching process (Section \ref{htsec}). Although we build on several prior works (discussed in respective sections), the obtained estimates are mostly new, are of independent interest, and form the backbone of our proofs.\\

\item In Section \ref{mardeg}, we recall some large deviations properties of a class of martingales (defined in \eqref{madef}) from \cite{BBpersistence} which have proved to be a crucial tool in analyzing the degree evolution of fixed vertices and competition between their degrees as the network grows.\\

\item In Section \ref{lowerunif}, a universal strategy to obtain budget lower bounds is described where the probability of the root being in the confidence set is lower bounded in terms of an event involving the root having an `unusually low' degree in the network (see Lemma \ref{lbpart1}).\\ 

\item The proofs of Theorems \ref{linearattach} and \ref{general} are completed in Sections \ref{linsec} and \ref{gensec}. The strategy discussed above, along with the obtained branching process estimates, is used in obtaining budget lower bounds. Obtaining budget upper bounds involves controlling the probabilities of two events: $A_n$ (see \eqref{A_n}) being the event that the root degree grows slower than a certain rate in the network size, and $E_n \cap A_n$, where $E_n$ (see \eqref{E_n}) is the event where the degree of the root is eventually exceeded by that of the $(n+1)$-th connecting vertex as the network grows. The former is bounded using branching process estimates and the later is bounded using the martingale large deviation results from Section \ref{mardeg}.\\

\item Finally, in Section \ref{nonsec}, Theorem \ref{nonpersistent} is proved by combining the estimates on hitting times, age of the maximal degree vertex and height of the branching process with Laplace transform asymptotics of the associated reproduction point process.
\end{enumerate}

\section{Branching process embedding }\label{BPpre}

\mn
\textbf{Branching process embedding in the case $m_i \equiv 1$.} A crucial technique to our proofs is an appropriate embedding of the discrete dynamics described in Section \ref{modeldef} into a continuous time branching process. 

\begin{definition}\label{CTBP}
A continuous time branching process (CTBP) driven by $f$, written as $\{BP_f(t): t\geq 0\}$, is defined to be a branching process started with one individual at time $t=0$ and such that every individual in this system has an offspring distribution that is an independent copy of the point process $\xi$ defined in \eqref{pprocess}.
\end{definition}
We will denote the number of individuals (size) of the branching process $BP_f(t)$ at time $t$ by $|BP_f(t)|$.

Let $BP^{(0)}_f(\cdot), BP^{(1)}_f(\cdot)$ be two independent branching processes driven by $\xi$ as in Definition \ref{CTBP} started from $v_0$ and $v_1$ respectively. Define $\widetilde{BP}_f(t):=BP^{(0)}_f(t)\cup BP^{(1)}_f(t)$. 
%We will let $BP_f(t)$ also denote the size of the branching process at time $t$ with some abuse of notation.

For $n\geq 2$, we label the vertices in $\widetilde{BP}_f(t)$ in the order of their arrivals into the system.  For $i=0,1$, $BP^{(i)}_f(t)$ can be viewed as a labelled tree rooted at $v_i$ with edges between parents and their offspring. Upon adding an edge between $v_0$ and $v_1$, $\widetilde{BP}_f(t)$ becomes a tree, and we assign its root to be $v_0$. Define the sequence of stopping times
\beq\label{Tdef}
T_l:=\inf\{ t\geq 0: |\widetilde{BP}_f(t)|=l+1\}, \quad\quad l\geq 1,
\eeq
that is, the arrival time of $v_l$ into the system.  Note that $T_1=0$. The following property follows from properties of Exponential distributions, and is the starting point of the Athreya--Karlin embedding \cite{athreya1968}.
%Note that we can couple $BP_f(t)$ and $\widetilde{BP}_f(t)$ through a time shift so that $T_m=\tau_m-\tau_1$.
\begin{lemma}\label{ctbpem}
Fix the attachment function $f$. Let $\{\GG_l : l\geq 1\}$ be the sequence of random trees constructed in Section \ref{modeldef} using $f$ as the attachment function and with $m_i \equiv 1$, and let $\{\widetilde{BP}_f(t):t\geq 0\}$ be the CTBP constructed as above. Then we have 
$$\{\widetilde{BP}_f(T_l), l\geq 1\} \overset{d}{=}\{ \GG_l: l\geq 1\}.$$
\end{lemma}

\noindent\textbf{Collapsed branching process embedding in the case $m_i \equiv m >1$.}
The embedding into a continuous time branching process does not carry over directly in the case $m_i \equiv m >1$. Nevertheless, we propose the following `continuous time interpolation' between $\GG_l$ and $\GG_{l+1}$ using an adaptive collapsing operation on a continuous time branching process. 
\begin{itemize}
\item Define the discrete time graph-valued process $\{CBP_f(l) : l \in \mathbb{N}\}$ as follows. $CBP_f(1)$ is the graph $\GG_1$ with two vertices $v_0, v_1$ connected by $m$ edges. Define $\tau_0=\tau_1 = 0$. 
\item Suppose for $l \ge 1$, we have constructed the random time $\tau_l$ and the process $\{CBP_f(j) : 1 \le j \le l\}$. To obtain $CBP_f(l+1)$, each individual $v \in CBP_f(l)$, say with degree $d(v,l)$ in $CBP_f(l)$, independently reproduces according to the point process 
$$\tilde \xi^{d(v,l)}(\cdot)\overset{d}{=}\xi_{d(v,l)-1}(t)$$
where, for any $k \in \mathbb{N}$, $\xi_{k-1}(t)$ is as defined in the line below \eqref{pprocess} with the attachment function $f^{(v)}(\cdot):=f(\cdot+k-1)$.
%$\tilde \xi^{d(v,l)}(\cdot)$ with successive reproduction times given by independent exponential random variables with rate $\{f(d(v,l)), f(d(v,l) + 1), \dots\}$. 

The newly born individuals do not reproduce until there are $m$ newborns since $\tau_l$. Define $\Theta_{l+1}$ to be the first time since $\tau_l$ when there are $m$ newly born individuals $v_{l+1,1}, \dots, v_{l+1,m}$ and set $\tau_{l+1} := \tau_l + \Theta_{l+1}$. Then we collapse the vertices $v_{l+1,1}, \dots, v_{l+1,m}$ to a single vertex $v_{l+1}$, and define $CBP_f(l+1)$ to be the graph obtained after this collapsing operation. Note that $d(v_{l+1}, l+1) = m$.
For any $j \in \mathbb{N}_0$, define the following process that represents the `continuous time' degree of $v_j$:
\begin{equation}\label{cbppoint}
\xi^{m,(v_j)}(t) := \sum_{l=j}^{\infty} \tilde \xi^{d(v_j,l)}((t \wedge \tau_{l+1}) - \tau_l), \ \ t \ge \tau_j,
\end{equation}
where we take $\tilde \xi^{d(v_j,l)}(s) =0$ for $s \le 0$. Note that $d(v_j,l) = m + \xi^{m,(v_j)}(\tau_l)$ for any $j \in \mathbb{N}_0$, $l \ge j$. 
%Also note that $\xi^{m,(v_j)}(t)\overset{d}{=}\xi_{m-1}(t-\tau_j)$, $t\geq \tau_j$.
\end{itemize}

The following lemma is easy to check using properties of Exponential distributions.

\begin{lemma}\label{cbp}
Fix an attachment function $f$ and an integer $m > 1$. Let $\{\GG_l : l\geq 1\}$ be the sequence of random trees constructed in Section \ref{modeldef} using $f$ as the attachment function and with $m_i \equiv m$, and let $\{CBP_f(l): l \in \mathbb{N}\}$ be the random graph process constructed as above. Then we have 
$$\{CBP_f(l), l\geq 1\} \overset{d}{=}\{ \GG_l: l\geq 1\}.$$
Moreover, the processes $\xi^{m,(v_j)}(\cdot \, + \, \tau_j), j \in \mathbb{N}_0,$ are independent and have the same distribution as the point process defined in \eqref{pprocess} with attachment function $f^{(m)}(k) := f(k + m-1)$ for $k \in \mathbb{N}$, that is, $\xi^{m,(v_j)}(t + \tau_j)\overset{d}{=}\xi_{m-1}(t)$, $t\geq 0$. 
\end{lemma}

The above construction is reminiscent of the collapsed continuous-time branching processes defined in \cite{garavaglia2018trees} (see also \cite{banerjee2023local}). However, the two constructions are different as the newly born vertices in $(\tau_l,\tau_{l+1})$ are not allowed to reproduce in the above construction, while they are in \cite{garavaglia2018trees}.

\mn
\textbf{Simple birth process with immigration.}
We describe one of the simple examples of continuous time branching processes which will be useful later, the simple birth process with immigration.  Let $\{Y_{\nu,\beta}(t): t\geq 0\}$ denote the simple birth process with birth rate $\nu>0$ and immigration rate $\beta>0$. Given the current population size $Y_{\nu,\beta}(t)$, the average instantaneous rate of growth is $\nu Y_{\nu,\beta}(t) + \beta$. More precisely,
\begin{align*}
P(Y_{\nu,\beta}(t+dt)-Y_{\nu,\beta}(t)=1| \FF(t))&=(\nu Y_{\nu,\beta}(t) + \beta)dt+o(dt),\\
P(Y_{\nu,\beta}(t+dt)-Y_{\nu,\beta}(t)\geq 2| \FF(t))&=o(dt),
\end{align*}
where $\{\FF(t):t\geq 0\}$ is the natural filtration of the process.

Set $Y_{\nu,\beta}(0)=1$ and define $M(\theta,t)=\mathbb{E}\left(e^{\theta Y_{\nu,\beta}(t)}\right)$ be the associated moment generating function. Then $M(\theta, t)$ satisfies 
$$\frac{\partial M}{\partial t}=\nu (e^{\theta}-1)\frac{\partial M}{\partial \theta}+\beta(e^{\theta}-1)M,$$
with initial condition $M(\theta,0)=e^\theta$ (see, e.g. \cite[Section 6.4.4]{Linda10}). The solution to the above differential equation is 
\beq\label{yulemgf}
M(\theta,t)=\frac{  e^{\theta}}{((1-e^\theta)e^{\nu t}+ e^\theta)^{1+\beta/\nu}},
\eeq
which is valid for $\theta < \log(e^{\nu t}/(e^{\nu t}-1))$.

\subsection{Asymptotics for population size and associated hitting times for CTBP}\label{ctbpas}
In this section we will collect some properties of the branching process introduced in Definition \ref{CTBP} that will be used later. 
Throughout this section, we will work under the assumptions (A1)-(A3) for $f$. The following theorem was recorded in \cite{jordan} using classical results of \cite{jagers1984growth,biggins1979continuity}.
\begin{theorem}[Theorem 3.5 in \cite{jordan}]\label{brlim}
Suppose the attachment function $f$ satisfies Assumptions (A1)-(A3). There exists a strictly positive random variable $W_\infty$ such that, with $\lambda^*$ being the Malthusian rate of growth parameter in \eqref{malthusian},
\beq
e^{-\lambda^*t}|BP_f(t)| \overset{a.s., \, \mathbb{L}^2}{\longrightarrow} W_\infty, \quad \text{ as }t\to\infty.
\eeq
Furthermore, $W_\infty$ is an absolutely continuous or singular continuous random variable (with respect to the Lebesgue measure) and is supported on all of $\RR_+$. 
%\textcolor{red}{Check Assumptions for this. Need $E(Y^2)< \infty$ for $L^2$ convergence, where $Y$ is as in centroid4.}
\end{theorem}
Taking $t= T_n$, where $T_n$ is defined in \eqref{Tdef}, the above theorem implies that 
\begin{equation}\label{tnlim}
T_n-\frac{1}{\lambda^*}\log n \overset{a.s.}{\to}-\frac{1}{\lambda^*} \log W_\infty \quad \text{ as }n\to\infty.
\end{equation}

The following lemma gives estimates for the tail behavior of the distribution of $W_{\infty}$ near zero.
\begin{lemma}\label{lefttail}
Let $W_\infty$ be as in Theorem \ref{brlim}. For any $a \in  [0, f(1)/ \lambda^*)$, there exists $C_a>0$ such that
$$P(W_\infty \leq x)\leq  C_a \, x^{a}, \ x > 0.$$

\end{lemma}
The proof of the above lemma relies on the crucial fact that the random variable $W_{\infty}$ satisfies the following recursive distributional equation:
\begin{equation}\label{RDE}
W_{\infty} \stackrel{d}{=} \sum_{i=1}^{\infty} e^{-\lambda^* \sigma_i} W_{\infty,i},
\end{equation}
where $\{\sigma_i\}_{i \ge 1}$ are the arrival times of the point process $\xi$ defined in \eqref{pprocess}, $\{W_{\infty,i}\}_{i \ge 1}$ are i.i.d copies of $W_{\infty}$ that are independent of $\{e^{-\lambda^* \sigma_i}\}_{i \ge 1}$. In particular, the moment generating function $\phi(t)=Ee^{-tW_\infty}, \, t \ge 0,$ of $W_{\infty}$ satisfies
\beq\label{funceq}
\phi(t)=E\prod_{i\geq 1}\phi(A_i t)
\eeq
where $A_i=e^{-\lambda^*\sigma_i}$. These recursive identities follow directly from Theorem \ref{brlim} and the dynamics of the branching process $BP_f$. The proof of Lemma \ref{lefttail} will require the following lemma from \cite{Liu2001}.
\begin{lemma}[\cite{Liu2001}]
Let $\phi:\RR_+\to \RR_+$ be a bounded function and let $A$  be a positive random variable such that for some $0<p<1$, $t_0\geq 0$ and all $t>t_0$,
$$\phi(t)\leq pE\phi(At).$$
If $pE(A^{-a})<1$ for some $0<a<\infty$, then $\phi(t)=O(t^{-a}) (t\to\infty)$ and $\int_0^\infty \phi(t)t^{a-1}dt<\infty$.
\end{lemma}

\begin{proof}[Proof of Lemma \ref{lefttail}]
Since $f$ satisfies Assumption (A2) and $f(k) \rightarrow \infty$ as $k \rightarrow \infty$ by Assumption (A1), Lemma 1 of \cite{rudas2007random} gives
$$
E(\hat{\xi}(\lambda^*) \log^+  \hat{\xi}(\lambda^*)) < \infty,
$$
where $\hat{\xi}(\lambda^*) := \int_0^\infty e^{-\lambda^* t}\xi(dt)$. Hence, by Proposition 1.1 of \cite{nerman1981convergence}, $P(W_\infty>0)=1$. Recalling the moment generating function $\phi$ of $W_{\infty}$, we obtain $\lim_{t\to\infty} \phi(t)=0$. In particular, there exists $t_0'>0$ such that for all $t>t_0'$, $\phi(t)< \frac{1}{2}$. Let $\delta \in (0,1)$ be arbitrarily fixed and define $N_\delta :=\sum_{i=1}^\infty 1_{\{ A_i>\delta\}}$, where we recall $A_i=e^{-\lambda^*\sigma_i}$. Note that $P(N_\delta <\infty)=1$, since $\lim_{i \rightarrow \infty} \sigma_i = \infty$ almost surely (which follows from the fact that, by the at-most-linear growth of $f$ required by Assumption (A1), $\sum_{i=1}^{\infty}\frac{1}{f(i)} = \infty$).

By the functional equation \eqref{funceq} and the fact that $\phi(A_i t)\leq 1$ for all $t$, we see that when $t>t_0'/\delta$,
$$
\phi(t) \leq E\left[ \phi(A_1t) 2^{-N_\delta + 1}\right]=p_{\delta}E\phi(\tilde{A}_1 t)
$$
where $p_{\delta}=E[2^{-N_\delta + 1}]$ and $\tilde{A}_1$ is a positive random variable whose distribution is determined by
$$E g(\tilde{A}_1)=\frac{1}{p_{\delta}} E\left[g(A_1)2^{-N_\delta + 1}\right]$$
for any bounded and measurable function $g$. 
%By the dominated convergence theorem,
%$$p_{\ep,\delta} \overset{ \delta \downarrow 0}{\to} E\ep^{N-1}1_{\{N\geq 1\}} \overset{\ep\downarrow 0}{\to} P(N=1).$$
Note that for $E A_1^{-a} = E e^{\lambda^* a \sigma_1} < \infty$ for any $a \in [0, f(1)/ \lambda^*)$. Choose and fix any such $a$. We can write
\begin{equation*}
p_{\delta}E\tilde{A}_1^{-a}=E [A_1^{-a}2^{-N_\delta + 1}].
\end{equation*}
Moreover, almost surely, $N_{\delta} \rightarrow \infty$ as $\delta \rightarrow 0$.
Hence, by the dominated convergence theorem,
\begin{align*}
\limsup_{\delta \rightarrow 0} E [A_1^{-a} 2^{-N_\delta + 1}] = 0.
\end{align*}
Hence, choosing $\delta \in(0,1)$ small enough, we have $p_{\delta}E\tilde{A}_1^{-a}<1$. Applying Lemma 3.2, we conclude that for any $a \in [0, f(1)/ \lambda^*)$, there exists $C_a'>0$ such that
$\phi(t) \le C_a' t^{-a}$ for any $t >0$.
It follows that for any $x>0$,
\begin{align*}
P(W_\infty \leq x)&\leq \inf_{\theta > 0} e^{\theta x}\phi(\theta)\leq C_a'\inf_{\theta > 0} \exp\left(\theta x- a \log \theta\right).
\end{align*}
Optimizing the above bound in $\theta$ leads to the bound claimed in the lemma.
\end{proof}

Recall $\{T_n\}_{n \ge 1}$ from \eqref{Tdef}. The following lemma gives tail estimates for the distribution of $T_n$.
\begin{lemma}\label{Tbounds}
Let $\{\gamma_n\}_{n\geq 1}$ be a non-decreasing sequence such that $\lim_{n\to\infty}\gamma_n=\infty$. 
%and \\$\limsup_{n\to \infty}\frac{\gamma_n}{\log n}=0$.  

\mn
(i) Suppose in addition $\limsup_{n\to \infty}\frac{\gamma_n}{\log n}=0$. For any $a \in  [0, f(1)/ \lambda^*)$, there exists $\tilde C_a>0$ such that for $n\geq 3$,
\beq\label{Tupper}
P(T_n\geq \frac{1}{\lambda^*}\log n+\gamma_n)\leq C  e^{- a \lambda^* \gamma_n}.
\eeq

\mn
(ii) For any $k\in\mathbb{N}$, there exists some $C_k>0$ so that for $n\geq1$,
\beq\label{Tlower}
P(T_n<\frac{1}{\lambda^*}\log n-\gamma_n)\leq C_ke^{-k\lambda^*\gamma_n}.
\eeq
\end{lemma}
\begin{proof}
We begin with the proof of \eqref{Tupper}. To simplify notation let $\sigma'_n=\frac{1}{\lambda^*}\log n+\gamma_n$. First note that
\begin{align*}
P(T_{n}\geq \sigma'_n)&=P(\widetilde{BP}_f(\sigma'_n)\leq n+1)=P(BP^{(0)}_f(\sigma'_n)+ BP^{(1)}_f(\sigma'_n)\leq n+1)\\
&\leq 2P(BP_f(\sigma'_n)\leq (n+1)/2)=2P\left(e^{-\lambda^* \sigma'_n}BP_f(\sigma'_n)\leq e^{-\lambda^* \sigma'_n}(n+1)/2\right),
\end{align*}
where the second line follows from the fact that $BP^{(0)}_f(\cdot)$ and $BP^{(1)}_f(\cdot)$ are i.i.d. copies of $BP_f(\cdot)$. Let $x_n= e^{-\lambda^* \sigma'_n}(n+1)/2$. According to Theorem 3.7 in \cite{jordan} there exist finite positive constants $C_1,C_2,\delta$ (depending on $f$) such that for any $t\geq 0$,
\beq\label{3.6err1}
P\left(\sup_{s\in [t,\infty)} |e^{-\lambda^* s}BP_f(s)-W_\infty|>e^{-\delta t}\right) \leq C_1e^{-C_2t}.
\eeq
Hence,
\begin{align*}
P\left( e^{-\lambda^* \sigma'_n}BP_f(\sigma'_n)\leq x_n\right)\leq P( |e^{-\lambda^* \sigma'_n}BP_f(\sigma'_n)-W_\infty|> e^{-\delta \sigma'_n})+P(W_\infty \leq x_n+e^{-\delta \sigma'_n}).
\end{align*}
By Lemma \ref{lefttail} for any $a \in  [0, f(1)/ \lambda^*)$, there exists $C_a>0$ such that for $n \ge 3$,
\beq\label{3.6err2}
P(W_\infty \leq x_n+e^{-\delta \sigma'_n})\leq  C_a(x_n +e^{-\delta \sigma'_n})^a.
\eeq
%Taking $\theta=\left( \max\{ x_n, e^{-\delta \sigma_n}\}\right)^{-1}$ gives,
%$$P(T_{n}\geq \sigma_n)\leq C' \left( \max\{ x_n, e^{-\delta \sigma_n}\}\right)^{\frac{f_*}{2\lambda^*}}+C_1e^{-C_2\sigma_n}.$$
When $n$ is sufficiently large, $e^{-\delta \sigma'_n} \ll x_n$ (this follows from the assumption that $\gamma_n = o(\log n)$). Hence, collecting \eqref{3.6err1} and \eqref{3.6err2}, there exists $\tilde C_a>0$ so that for $n\geq 3$,
$$
P(T_n\geq \frac{1}{\lambda^*}\log n+\gamma_n)\leq  \tilde C_a  e^{-a \lambda^* \gamma_n}.
$$
It remains to prove \eqref{Tlower}. Let $\sigma''_n=\frac{1}{\lambda^*}\log n-\gamma_n$. Note that
\begin{align*}
P(T_n<\sigma''_n)&= P(BP^{(0)}_f(\sigma''_n)+ BP^{(1)}_f(\sigma''_n)\geq n+1)\leq 2P(BP_f(\sigma''_n)\geq n/2)\\
%&P\left(BP^{(_f(\frac{1}{\lambda^*}\log n-\gamma_n/2)\geq n\right)+P(\tau_1\geq \gamma_n/2)\\
&=2P(e^{-\lambda^* \sigma''_n}BP_f(\sigma''_n)\geq e^{-\lambda^* \sigma''_n} \cdot n/2)\\
&=2P(e^{-\lambda^* \sigma''_n}BP_f(\sigma''_n)\geq e^{\lambda^*\gamma_n}/2)\\
&\leq 2^{k+1}e^{-k\lambda^* \gamma_n} E\left[ \left(\sup_{t\geq 0} e^{-\lambda^* t}|BP_f(t)|\right)^k\right]
\end{align*}
for any positive integer $k$.
Recall that $\{\sigma_i: i\geq1\}$ denote the arrival times of the point process $\xi$ defined in \eqref{pprocess}. By Theorem 4 in \cite{Mori19}, if for any positive integer $k$, $E\left(\sum_{i=1}^\infty e^{-\lambda^*\sigma_i}\right)^k<\infty$ and $E(|BP_f(t)|^k)<\infty$ for every $t\geq 0$, then 
$$
C_k :=  2^{k+1}E\left[ \left(\sup_{t\geq 0} e^{-\lambda^* t}|BP_f(t)|\right)^k\right]<\infty.
$$
To complete the proof it suffices to check the two conditions hold. By Proposition 5.7 in \cite{jordan}, the random variable $Y := \sum_{i=1}^\infty e^{-\lambda^*\sigma_i}$ satisfies $E(e^{\delta Y})< \infty$ for some $\delta>0$. In particular, $E\left(\sum_{i=1}^\infty e^{-\lambda^*\sigma_i}\right)^k<\infty$ for any $k\in \mathbb{N}$. Moreover, by Assumption (A1), $BP_f(t)$ is dominated by a Yule process with a linear attachment function. Hence, by \eqref{yulemgf}, it is clear that $E(|BP_f(t)|^k)<\infty$ for every $t\geq 0$. 
\end{proof}

%\begin{lemma}\label{Tbounds}
%Let $\{\gamma_n\}_{n\geq 1}$ be a non-decreasing sequence such that $\lim_{n\to\infty}\gamma_n=\infty$ and \\$\limsup_{n\to \infty}\frac{\gamma_n}{\log n}=0$. Then for any $k\in\mathbb{N}$, there exists some $C_k>0$ such that
%$$P(T_n<\frac{1}{\lambda^*}\log n-\gamma_n)\leq C_ke^{-k\lambda^*\gamma_n}.$$
%\end{lemma}
%\begin{proof}
%Let $\sigma_n=(\log n)/\lambda^*-\gamma_n$. Note that
%\begin{align*}
%P(T_n<\frac{1}{\lambda^*}\log n-\gamma_n)&=P(\tau_n<\frac{1}{\lambda^*}\log n-\gamma_n+\tau_1)\\
%&\leq P\left(BP_f(\frac{1}{\lambda^*}\log n-\gamma_n/2)\geq n\right)+P(\tau_1\geq \gamma_n/2)\\
%&=P(e^{-\lambda^* \sigma_n}BP_f(\sigma_n)\geq e^{-\lambda^* \sigma_n} \cdot n)\\
%&=P(e^{-\lambda^* \sigma_n}BP_f(\sigma_n)\geq e^{\lambda^*\gamma_n})\\
%&\leq e^{-k\lambda^* \gamma_n} \left( \sup_{t\geq 0} e^{-\lambda^* t}\|BP_f(t)\|_k\right)^k
%\end{align*}
%By Theorem 4 in \cite{Mori19}, if $\|_\alpha \xi(\infty)\|_k$ and $\|BP_f(t)\|_k<\infty$ for every $t\geq 0$, then 
%$$C_k\equiv  \left( \sup_{t\geq 0} e^{-\lambda^* t}\|BP_f(t)\|_k\right)^k<\infty.$$
%As our process is dominated by a Yule process with a linear attachment function, it is clear that $\|BP_f(t)\|_k<\infty$ for every $t\geq 0$. Note that 
%$_\alpha \xi(\infty)=\sum_{i=1}^{\infty} e^{-\alpha \sigma_i}$, which is exactly the $Y$ defined in (5.2) \cite{jordan}.  Propostion 5.7 implies that $\|_\alpha \xi(\infty)\|_k<\infty$ for any $k\in \mathbb{N}$. 
%\end{proof}

\subsection{Asymptotics for age of maximal degree vertex in CTBP}\label{agesec}

The proof of Theorem \ref{nonpersistent} relies on a precise quantification of the tail probabilities of the age of the (oldest) vertex with maximal degree in a continuous time branching process. This is achieved in this subsection.

For $a\geq 0,i\in\mathbb{N}$, let $B^{(i)}_a$ be the birth time of the $i$-th individual born at or after time $a$ and let $\xi^{(i)}$ denote the point process of times when this individual reproduces, measured relative to its own birth time (i.e. $\xi^{(i)}(s')=k$ if individual $i$ has $k$ children at time $s' + B^{(i)}_a$). For $s\geq 0$, let $D^{(i)}_a(s) := \Phi_1(\xi^{(i)}(s))$. Again, note that for any $a\geq 0$, $\{D^{(i)}_a(\cdot): i\in \mathbb{N}\}$ are i.i.d. processes, each having the same distribution as $D_0^{(1)}(\cdot)$. For any $0\leq a<b$, let $n[a,b]$ denote the number of individuals born to the branching process in the time interval $[a,b]$.

For any $0\leq a<b\leq c$ with $n[a,b]>0$, let  
\beq
D^{max}_{a,b}(c):=\sup\{ D^{(i)}_a(c-B^{(i)}_a):i\in \mathbb{N} \text{ such that }B^{(i)}_a\leq b\}
\eeq
denote the maximum of the degrees of vertices born in the time interval $[a,b]$ at time $c$. We define $D^{max}_{a,b}(c) =0$ if $c < b$ or $n[a,b] = 0$.
Define
\beq\label{maxbirth}
\II^*_c(t):=\sup \{B_0^{(i)}: D^{(i)}_0(t-B_0^{(i)})=D^{max}_{0,t} (t)\}.
\eeq
Notice that this definition is different from that of $\II^*_c(t)$ in \cite{BBpersistence} (there is an $\inf$ in place of $\sup$ in \cite{BBpersistence}). However, technically, they can be analyzed using similar tools. The following is the main result of this subsection.

\begin{lemma}\label{maxposition}
Suppose $\KK(\cdot) := \Phi_2 \circ \Phi_1^{-1}(\cdot)$ satisfies the continuity assumptions in Theorem \ref{nonpersistent}. For any $\ep>0$,
$$ P\left( \II^*_c(T_n) > \left(\frac{\lambda^*}{2}+\ep\right)\KK\left(\frac{1}{\lambda^*}\log n\right)\right) \leq C\exp\left(-C'\KK\left(\frac{1}{\lambda^*}\log n\right)\right), \ n \ge 3,$$
for some positive constants $C,C'$.
\end{lemma}
The proof of Lemma \ref{maxposition} relies on quantitative versions of Lemmas 7.12-7.14 in \cite{BBpersistence} described in the following three lemmas.

\begin{lemma}\label{lem7.12}
Assume that the attachment function $f$ satisfies Assumptions (A1)-(A3). For any $0\leq a<b$, there exists $\delta_1>0$ such that for any $\delta\in (0, \delta_1)$, there exist $C,C'>0$ that depend on $f, \delta,a,b$ such that 
$$
P(n[at,bt]\leq e^{(\lambda^*-\delta)bt})\leq Ce^{-C't}, \ \ P(n[0,bt]>e^{(\lambda^* +\delta)bt}) \le Ce^{-\delta b t}.
$$
\end{lemma}
%\begin{remark}
%To figure out the constant $C_2$ need to go through proof of Theorem 3.7 in \cite{jordan}.
%\end{remark}

\begin{proof}
Observe that, for any $\delta>0$,
\begin{align}\label{nest}
P(n[at,bt]\leq e^{(\lambda^*-\delta)bt})&\leq P(n[0,at] > e^{(\lambda^*-\delta)bt}) + P(n[0,bt]\leq 2e^{(\lambda^*-\delta)bt}).
\end{align}
To estimate the first term in the bound \eqref{nest}, note that for any $\delta\in (0, \lambda^*(b-a)/(b+a))$,
\begin{align}\label{a1}
P(n[0,at] > e^{(\lambda^*-\delta)bt}) \le P(n[0,at] > e^{(\lambda^* + \delta)at}) \le \sup_{s \ge 0}E\left(e^{-\lambda^*s}n[0,s]\right) e^{-\delta at},
\end{align}
where the supremum above is finite by \cite[Proposition 2.2]{nerman1981convergence} and Assumption (A2).

To estimate the second term in the bound \eqref{nest}, recall that by Theorem \ref{brlim}, $e^{-\lambda^* bt}n[0,bt]$ converges almost surely to a random variable $W_{\infty}$ as $t \rightarrow \infty$. Write
\begin{align}\label{b1}
P(n[0,bt]\leq 2e^{(\lambda^*-\delta)bt}) &= P(e^{-\lambda^* bt}n[0,bt]\leq 2e^{-\delta bt})\nonumber\\
& \le P\left(\left| e^{-\lambda^* bt}n[0,bt] - W_{\infty}\right| > e^{-\delta b t}\right) + P(W_{\infty} \leq 3e^{-\delta bt}).
\end{align}
By Theorem 3.7 in \cite{jordan}, there exist finite positive constants $C_1,C_2,\delta_0$ (depending on $f$) such that for any $t\geq 0$,
$$
P\left(\sup_{s\in [t,\infty)} |e^{-\lambda^* s}n[0,s] -W_\infty|>e^{-\delta_0 t}\right) \leq C_1e^{-C_2t}.
$$
Hence, for any $\delta \in (0, \delta_0/b)$,
$$
P\left(\left| e^{-\lambda^* bt}n[0,bt] - W_{\infty}\right| > e^{-\delta b t}\right) \le P\left(\sup_{s\in [t,\infty)} |e^{-\lambda^* s}n[0,s] -W_\infty|>e^{-\delta_0 t}\right) \leq C_1e^{-C_2t}.
$$
Moreover, using Lemma \ref{lefttail} with $a = f(1)/(2\lambda^*)$, we obtain $C'>0$ such that
$$
P(W_{\infty} \leq 3e^{-\delta bt}) \le C' e^{-f(1)\delta b t/(2\lambda^*)}.
$$
Using the above two bounds in \eqref{b1},
\begin{equation}\label{b2}
P(n[0,bt]\leq 2e^{(\lambda^*-\delta)bt}) \le C_1e^{-C_2t} +  C' e^{-f(1)\delta b t/(2\lambda^*)}.
\end{equation}
The first bound in the lemma, with $\delta_1 := [\lambda^*(b-a)/(b+a)] \wedge [\delta_0/b]$, now follows upon using \eqref{a1} and \eqref{b2} in \eqref{b1}. The second bound follows similarly as \eqref{a1}.
\end{proof}

Recall $\KK(t) :=\Phi_2\circ \Phi_1^{-1}(t), t\geq 0.$
\begin{lemma}\label{lem7.13}
Assume $\Phi_2(\infty)=\infty$ and the attachment function $f$ satisfies Assumption (A1). Also assume that 
\begin{equation}\label{contK}
\lim_{\delta \downarrow 0}\limsup_{t\to\infty} \frac{\KK((1+\delta)t)}{\KK(t)}=1.
\end{equation}
Then for any $0\leq a<b $ and any $\eta\in (0, \sqrt{2\lambda^* b}/2)$, there exists $C,C'>0$ (depending on $a,b,\eta$) such that 
\begin{align}\label{ev}
P\left(D^{max}_{a\KK(t),b\KK(t)}(t-\eta \KK(t)) < t-b\KK(t)+(\sqrt{2\lambda^* b}-2\eta)\KK(t)\right)&\leq Ce^{-C'\KK(t)}, \nonumber\\
P\left(D^{max}_{a\KK(t),b\KK(t)}(t+\eta \KK(t)) > t-a\KK(t)+(\sqrt{2\lambda^* b}+2\eta)\KK(t)\right)&\leq Ce^{-C'\KK(t)}.
\end{align}
\end{lemma}

\begin{proof}
The proof follows from that of Lemma 7.13 in \cite{BBpersistence}. The major difference is that Lemma \ref{lem7.12} above is used to produce quantitative bounds for the events in \eqref{ev} (in Lemma 7.13 of \cite{BBpersistence}, the probabilities were shown to converge to zero without quantitative bounds). 

Recall $\delta_1$ from Lemma \ref{lem7.12} for the given $a,b$. Take $\ep \in (0,1), \, \delta \in (0, \delta_1)$ small enough such that 
\begin{equation}\label{ed}
\sqrt{2\lambda^* b}-\eta < \frac{\sqrt{2(\lambda^*-\delta)b}}{(1+\ep)^2}, \  \ \sqrt{2\lambda^* b} + \eta >  \frac{\sqrt{2(\lambda^* + \delta)b}}{(1- \ep)^2}.
\end{equation}
For any $z\geq -\eta/2, t\geq t_2(z)$ for some $t_2=t_2(z)$ depending on $z$, estimate (7.52) in \cite{BBpersistence} produces the following estimate for any $i \in \mathbb{N}$:
$$
P\left(D_{a\KK(t)}^{(i)}(t - \eta \KK(t) - b\KK(t)) < t-b\KK(t)+z\KK(t)\right)
\le 1-\exp\left( -\frac{(1+\ep)^4(z+\eta)^2\KK(t)}{2}\right).
$$
Using this and the first bound in Lemma \ref{lem7.12}, we obtain
\begin{align*}
&P\left(D^{max}_{a\KK(t),b\KK(t)}(t-\eta \KK(t))< t-b\KK(t)+z\KK(t)\right)\\
\leq & P( n[a\KK(t),b\KK(t)]\leq e^{(\lambda^* -\delta)b\KK(t)})\\
&\qquad + P\left(D_{a\KK(t)}^{(i)}(t - \eta \KK(t) - b\KK(t)) < t-b\KK(t)+z\KK(t) \ \text{for all} \ i \le \lfloor e^{(\lambda^* -\delta)b\KK(t)} \rfloor + 1\right)\\
\leq & C_1e^{-C_2\KK(t)} + \left(1-\exp\left( -\frac{(1+\ep)^4(z+\eta)^2\KK(t)}{2}\right)\right)^{e^{(\lambda^*-\delta)b\KK(t)}}\\
\leq& C_1e^{-C_2\KK(t)}+\exp\left(-\exp\left((\lambda^*-\delta)b\KK(t)-\frac{(1+\ep)^4(z+\eta)^2\KK(t)}{2} \right)\right).
\end{align*}
Note that, by \eqref{ed} and as $\eta\in (0, \sqrt{2\lambda^* b}/2)$, $\frac{\sqrt{2(\lambda^*-\delta)b}}{(1+\ep)^2}-\eta > \sqrt{2\lambda^* b}- 2\eta >0$. For any $-\eta/2\leq z< \frac{\sqrt{2(\lambda^*-\delta)b}}{(1+\ep)^2}-\eta$ we have 
$$-\exp\left((\lambda^*-\delta)b\KK(t)-\frac{(1+\ep)^4(z+\eta)^2\KK(t)}{2} \right)\leq -\exp( C_{z,\eta}\KK(t))$$
for some positive constant $C_{z,\eta}$ depending only on $z,\eta$. Hence, taking $z=\sqrt{2\lambda^* b}-2\eta$,
\begin{align*}
&P\left(D^{max}_{a\KK(t),b\KK(t)}(t-\eta \KK(t))< t-b\KK(t)+z\KK(t)\right)\\
&\leq C_1e^{-C_2\KK(t)}+\exp(-\exp(C_{z,\eta}\KK(t)))\\
&\leq C_3e^{-C_2\KK(t)}
\end{align*}
for sufficiently large $t$, where we have used $\KK(t) = \Phi_2\circ \Phi_1^{-1}(t) \rightarrow \infty$ as $t \rightarrow \infty$ (which is a consequence of $\Phi_1(\infty)=\Phi_2(\infty) = \infty$). Note that the constants $C_2,C_3$ can depend on $b, \eta$. This proves the first bound in \eqref{ev}.

The second bound in \eqref{ev} follows similarly upon using the second bound in Lemma \ref{lem7.12} and following the calculations (7.57)-(7.60) in \cite{BBpersistence}:
\begin{align*}
&P\left(D^{max}_{a\KK(t),b\KK(t)}(t+\eta \KK(t))> t-a\KK(t)+w\KK(t)\right)\\
\leq& P(n[0,b\KK(t)]>e^{(\lambda^*+\delta)b\KK(t)}) +\exp\left( (\lambda^*+\delta)b\KK(t)-\frac{ (w-\eta)^2(1-\ep)^4\KK(t)}{2}\right)\\
\leq &Ce^{-\delta b\KK(t)} +\exp(-C_{w,\eta}\KK(t))
\end{align*}
where $w=\sqrt{2\lambda^* b}+2\eta$.
\end{proof}

\begin{lemma}\label{lem7.14}
Assume $\Phi_2(\infty)=\infty$. Moreover, assume that there exists positive constants $t'$, $D$ such that $\KK(3t)\leq D\KK(t)$ for all $t\geq t'$. Then there exists $A_0>0$ such that 
$$P( D^{max}_{A_0\KK(t),t}(t+\KK(t))>t)\leq Ce^{-C'\KK(t)}$$
for some constants $C,C'>0$.
\end{lemma}
\begin{proof}
Define the event 
$$\EE_A :=\{ \sup_{t\geq 0} e^{-\lambda^* t}n[0,t]\leq A\}.$$
By Corollary 3.8 in \cite{jordan},
$$P(\EE^c_A)\leq C\log A/A^2.$$
By (7.63) of \cite{BBpersistence}, we have some $j_0\geq 2$ and $t^*>0$ such that for all $t \ge t^*$ and $ j_0 \le j \le 1 + (2t/\KK(t))$,
$$P(D^{max}_{j\KK(t),(j+1)\KK(t)}(t+\KK(t))>t, \EE_A)\leq A\exp\left(-\frac{(j-1)^2}{8D}\KK(t)\right).$$
By (7.64) of \cite{BBpersistence}, we have for any $t\geq t^*$,
\begin{align*}
P(D^{max}_{j_0\KK(t),t}(t+\KK(t))>t, \EE_A)&\leq  \sum_{j=j_0}^{1+\lfloor t/\KK(t) \rfloor} P(D^{max}_{j\KK(t),(j+1)\KK(t)}(t+\KK(t))>t, \EE_A)\\
& \leq \sum_{j=j_0}^{1+\lfloor t/\KK(t) \rfloor} A\exp\left(-\frac{(j-1)^2}{8D}\KK(t)\right)\leq C A\exp\left(-\frac{(j_0-1)^2}{8D}\KK(t)\right),
\end{align*}
where $C$ is a positive constant not depending on $t,A$.
Finally,
\begin{align*}
P(D^{max}_{j_0\KK(t),t}(t+\KK(t))>t) &\leq P(D^{max}_{j_0\KK(t),t}(t+\KK(t))>t, \EE_A)+P(\EE^c_A)\\
&\leq C A\exp\left(-\frac{(j_0-1)^2}{8D}\KK(t)\right)+C\log A/A^2.
\end{align*}
The lemma now follows upon taking $A=\exp\left(\frac{(j_0-1)^2}{16D}\KK(t)\right)$.
\end{proof}

\mn
\textit{Proof of Lemma \ref{maxposition}.}
Write $u^*:= \lambda^*/2$. Take any $p>0$, whose value will be appropriately chosen later. Let $t^+_n= \frac{1}{\lambda^*}\log n+p\KK(\frac{1}{\lambda^*}\log n)$ and $t^{-}_n= \frac{1}{\lambda^*}\log n-p\KK(\frac{1}{\lambda^*}\log n)$. Note that 
\begin{align*}
&P\left( \II^*_c(T_n)\leq  (u^*+\ep)\KK(\frac{1}{\lambda^*}\log n)\right)\\
\geq& P\left(\II^*_c(s)\leq   (u^*+\ep)\KK(\frac{1}{\lambda^*}\log n) \text{ for all $s\in[t^-_n,t^+_n]$},T_n\in[t^{-}_n, t^+_n]\right).
\end{align*}
Hence,
\begin{align}\label{in0}
&P\left( \II^*_c(T_n) >  (u^*+\ep)\KK(\frac{1}{\lambda^*}\log n)\right)\nonumber\\
\leq &P\left(\II^*_c(s) >   (u^*+\ep)\KK(\frac{1}{\lambda^*}\log n) \text{ for some $s\in[t^-_n,t^+_n]$}\right)+P(T_n>t^+_n)+P(T_n<t^-_n).
\end{align}
 Consider the function 
 $$H(u)=-u+\sqrt{2\lambda^* u}, \quad u\geq 0.$$
 This function has a unique maximum at $u=u^*:=\lambda^*/2$ with maximum value $H(u^*)=\lambda^*/2=u^*$. Recall $A_0$ from Lemma \ref{lem7.14} and take $A_1:=\lambda^*+A_0$.
 
 Let $\ep_0\in (0,\lambda^*/4)$ be small enough that $\max\{H(u^*)-H(u^*-\ep), H(u^*)-H(u^*+\ep)\}< \sqrt{2\lambda^*(u^*-\ep)}$ for all $\ep<\ep_0$. Fix any $\ep\in(0,\ep_0)$. Let $\zeta:=\min\{H(u^*)-H(u^*-\ep), H(u^*)-H(u^*+\ep)\}$. 
Partition the set $(u^*+\ep,A_1]$ into disjoint intervals $[u_i,v_i)$, $i\in \mathbb{J}$, of mesh $\zeta/3$. Let $p=\min\{1, \zeta/12\}$. Then for any $t>0$,
\begin{align*}
&P( \II^*_c(s) \leq (u^*+\ep)\KK(t) \text{ for all $s\in[t-p\KK(t), t+p\KK(t)]$})\\
&\geq P\bigg( D^{max}_{(u^* - \ep)\KK(t),u^*\KK(t)}(s)\geq t+(H(u^*)-\zeta/3)\KK(t),\\
&\quad D^{max}_{u_i\KK(t),v_i\KK(t)}(s)\leq t+(H(v_i)+\zeta/2)\KK(t) \text{ for all $i \in \mathbb{J}$},\\
&\quad D^{max}_{A_1\KK(t),t}(s)\leq t, \text{ for all $s\in[t-p\KK(t), t+p\KK(t)]$}\bigg).
\end{align*}
Hence, taking complements and applying the union bound,
\begin{align*}
&P\left(\II^*_c(s) >   (u^*+\ep)\KK(t) \text{ for some $s\in[t-p\KK(t), t+p\KK(t)]$}\right)\\
 &\le P\bigg(D^{max}_{(u^* - \ep)\KK(t),u^*\KK(t)}(t-p\KK(t)) < t+(H(u^*)-\zeta/3)\KK(t)\bigg)\\
 &\quad + \sum_{i \in \mathbb{J}} P\left(D^{max}_{u_i\KK(t),v_i\KK(t)}(t+p\KK(t)) > t+(H(v_i)+\zeta/2)\KK(t)\right)\\
 &\quad + P\left(D^{max}_{A_1\KK(t),t}(t+p\KK(t)) > t\right).
\end{align*}
The first term in the product can be estimated using Lemma \ref{lem7.13} by taking $\eta=\zeta/6, a=u^*-\ep,b=u^*$, and the sum can be estimated using the same lemma by taking $\eta=\zeta/12, a=u_i,b=v_i$ for each $i\in \mathbb{J}$ and noting that the cardinality of $\mathbb{J}$ does not depend on $t$. The last term can be estimated from Lemma \ref{lem7.14} upon noting that $A_1>A_0$. Hence, we obtain upon putting $t = \frac{1}{\lambda^*} \log n$,
\beq\label{term1}
P\left(\II^*_c(s) >   (u^*+\ep)\KK(\frac{1}{\lambda^*}\log n) \text{ for some $s\in[t^-_n,t^+_n]$}\right)\leq Ce^{-C'\KK(\frac{1}{\lambda^*}\log n)}.
\eeq
To estimate the last two terms in \eqref{in0}, we will use Lemma \ref{Tbounds} with $\gamma_n := \KK(\frac{1}{\lambda^*}\log n)$. Note that, as $\Phi_2(\infty) = \infty$, $\lim_{t \rightarrow \infty} \KK(t) = \infty$. Moreover, as $f(n) \rightarrow \infty$ as $n \rightarrow \infty$ by Assumption (A1), for any $\delta>0$, there exists $n_{\delta} \in \mathbb{N}$ such that $f(n) > \delta^{-1}$ for all $n \ge n_{\delta}$. Hence, for $n > n_{\delta}$,
$$
\Phi_2(n) \le \sum_{k=1}^{n_{\delta}} \frac{1}{f^2(k)} + \delta\sum_{k=n_{\delta} + 1}^{n} \frac{1}{f(k)} \le  \sum_{k=1}^{n_{\delta}} \frac{1}{f^2(k)} + \delta\Phi_1(n).
$$
In particular, $\limsup_{n \rightarrow \infty}\frac{\Phi_2(n)}{\Phi_1(n)} \le \delta$. As $\delta>0$ is arbitrary, $\limsup_{n \rightarrow \infty}\frac{\Phi_2(n)}{\Phi_1(n)} = 0$.
Hence,
\begin{equation}\label{ktozero}
\lim_{t \rightarrow \infty} \frac{\KK(t)}{t} = \lim_{t \rightarrow \infty}\frac{\Phi_2 \circ \Phi_1^{-1}(t)}{\Phi_1 \circ \Phi_1^{-1}(t)}=0.
\end{equation}
Therefore, $\gamma_n := \KK(\frac{1}{\lambda^*}\log n)$ satisfies the hypotheses of Lemma \ref{Tbounds}, and thus
\beq\label{term2}
P(T_n<t^-_n)+P(T_n> t^+_n)\leq Ce^{-C'\KK(\frac{1}{\lambda^*}\log n)}.
\eeq
Using \eqref{term1} and \eqref{term2} in \eqref{in0}, we have 
$$
P\left( \II^*_c(T_n) > (u^*+\ep)\KK(\frac{1}{\lambda^*}\log n)\right)\leq Ce^{-C'\KK(\frac{1}{\lambda^*}\log n)}.
$$
\qed

\subsection{Height of branching processes}\label{htsec}

Recall the continuous time branching process $BP_f(t)$, viewed as a labelled tree (vertices labelled in order of arrival). Let $H_t$ denote the height of this tree at time $t$, that is, the maximal distance of a vertex in $BP_f(t)$ from the root. We will need the following tail estimate for the height from \cite{HJarxiv}.

\mn
\begin{lemma}[Lemma 13.15 of \cite{HJarxiv}]\label{height}
 For every $r>0$, there exists $c_r$ such that, for large $t$,
$$P(H_t\geq c_r t)\leq e^{-rt}.$$
\end{lemma}

%\mn
%\textbf{Profile.} The distribution of depths of the vertices is called the \textit{profile} of the tree in \cite{HJarxiv}.
%Let in this section, for a rooted tree $\TT$ and a real number $s$,
%$$n_{\leq s}(\TT):=\{ v \in \TT: h(v)\leq s\},$$
%where $h(v)$ denotes the depth of $v$.
%
%\mn
%\textbf{Theorem 13.43 in \cite{HJarxiv}.}
%{\it Under some assumptions, as $n\to\infty$, a.s.,
%\beq\label{profile}
%n_{\leq x\log n}(\TT_n)=n^{\tilde{\alpha}^*(\lambda x)/\lambda+o(1)},
%\eeq
%uniformly in $x\geq x_0$ for every $x_0>0$; if further $\underline{\lambda}\leq 0$, then \eqref{profile} holds uniformly for all $x\geq 0$.
%}
%
% By Theorem 13.43, as $n\to\infty$, a.s.,
%$$n_{\leq x\log n}(\TT_n)=n^{\tilde \alpha^*(\lambda x)/\lambda +o(1)},$$
%uniformly for all $x\geq 0$, where 
%$$\tilde \alpha^*(x)=\inf_{0\leq \theta\leq \lambda} \{ x\log \hat{\mu}(\theta)+\theta\}.$$

\section{Martingales tracking degrees}\label{mardeg}
For $A\in \mathbb{N}_0$, recall the definition of $f_A(\cdot)$ and $\xi_A(\cdot)$ from \eqref{pprocess}. Define the process
\beq\label{madef}
M_A(t):=\sum_{i=1}^{\xi_A(t)}\frac{1}{f_A(i)}-t, \ t\geq 0,
\eeq
where we recall our convention that $\sum_{i=1}^0\frac{1}{f_A(i)} =0$. This process turns out to be the crucial technical tool used in \cite{BBpersistence}, as well as the current article, to track degrees of individual vertices in the random graph process. We collect some properties of this process, proved in \cite{BBpersistence}. For $A \in \mathbb{N}_0$, write $f_*(A) := \inf_{i \ge A + 1} f(i)$.

\mn
\begin{lemma}[Lemmas 7.1, 7.3 and 7.4 in \cite{BBpersistence}]\label{mprop}
For any $A \in \mathbb{N}_0$, the process $M_A(\cdot)$ defined in \eqref{madef} is a martingale with respect to its natural filtration.

If $\Phi_2(\infty)<\infty$, then there exists positive constants $x_0,x_0',x_0'',C_1,C_2$ (independent of $A$) such that for any $A\in \mathbb{N}_0$,
\beq
P\left(\inf_{s<\infty} M_A(s)\leq -x\right) \leq 
\begin{cases}
C_1\exp(-C_2x^2) &\text{ if } x_0\leq x<x_0'\sqrt{f_*(A)},\\
C_1\exp(-C_2\sqrt{f_*(A)}x) &\text{ if }x\geq x_0'\sqrt{f_*(A)}.
\end{cases}
\eeq 
and 
\beq
P\left(\sup_{s<\infty} M_A(s)\geq x\right) \leq C_1\exp(-C_2x^2)\quad \text{ for all }x\geq x_0''.
\eeq
Moreover, if $\Phi_2(\infty) < \infty$, $M_A(t) \rightarrow M_A(\infty)$ almost surely, and the law of $M_A(\infty)$ is absolutely continuous with respect to the Lebesgue measure.
\end{lemma}

\section{A universal approach for budget lower bounds} \label{lowerunif}
Our general approach for obtaining lower bounds on $K_{\Psi}(\ep)$ is to connect $\liminf_{n\to\infty} P(v_0\notin H_{\lfloor \beta K/2\rfloor,\Psi}(\GG_n))$, for some $\beta \in (0,1)$ and any $K \ge 2$, to the probability of the event $\AA_{\beta, d^*,K}$ defined below. 
For $K\geq 1$, $d^* \ge m$, $\beta \in (0,1)$, define
\begin{align*}
\AA_{\beta, d^*, K} &=\{ v_0 \text{ has degree $\leq d^*$ in $\GG_{K}$ and there are at least $\lfloor \beta K\rfloor$}\\
& \qquad \qquad \text{ other vertices with degree $\ge d^*$ in $\GG_{K}$}\}.
\end{align*}
Our proofs for lower bounds on the budget will involve lower bounds on $P(\AA_{\beta,d^*,K})$ for appropriate choices of $\beta, d^*$. This, by Lemma \ref{lbpart1}, will produce lower bounds on $K_{\Psi}(\ep)$.

Following Definition \ref{budget}, for $1 \le j \le n$ we will denote by $H_{j,\Psi}(\GG_n)$ the set of $j$ vertices in $\GG_n$ with the largest $\Psi$ values (ties being broken arbitrarily). As $\Psi$ is taken to be the degree centrality measure here, $H_{j,\Psi}(\GG_n)$ is the set of $j$ vertices with the largest degrees in $\GG_n$ (ties being broken arbitrarily).

\begin{lemma}\label{lbpart1}
Suppose $m_i \equiv m \ge 1$ and $\Phi_2(\infty)<\infty$. For any $d^* \ge m$, $\beta \in (0,1]$ and $K\geq 4 \beta^{-1}$,
$$\liminf_{n\to\infty} P(v_0\notin H_{\lfloor \beta K/2\rfloor,\Psi}(\GG_n))\geq \frac{1}{2}P(\AA_{\beta,d^*,K}).$$
\end{lemma}
\begin{proof}
We will use the construction and terminology of the collapsed branching process embedding defined in Section \ref{BPpre}. Let $\FF_K$ be the natural (stopped) filtration associated with the continuous time description of the collapsed branching process up to stopping time $\tau_K$. For $0 \le i \le K, t \ge 0$, write $d'_i(t) := \xi^{(m),v_i}(t + \tau_K - \tau_i)$ to denote the `continuous time' degree of $v_i$  at time $t+\tau_K$, conditioned on the filtration $\FF_K$. By definition we see that $d'_i(0)$ represents the degree of $v_i$ in $CBP_f(K)$ (which is known given $\FF_K$) for each $0\leq i\leq K$. 

Let $\mathcal{U}_K$ denote the set of indices $i \in \{1, \dots,K\}$ for which the degree of $v_i$ in $\GG_K$ is at least $d^*$. 
Conditioned on the filtration $\FF_K$, and on the event $\AA_{\beta, d^*, K}$, $|\mathcal{U}_K| \ge \lfloor \beta K\rfloor$, and the degree of $v_0$ is at most $d^*$. It follows from the construction of the collapsed branching process embedding that $\{d'_i(t)\}_{0\leq i\leq K}$ are independent. So, we can construct a collection of i.i.d. processes $\{\zeta^{(i)}_{d^*}(t): t\geq 0\}_{i \in \mathcal{U}_K \cup\{0\}}$, each having the same law as $\zeta_{d^*}(t) := \xi_{d^*-1}(t)+d^*, t \ge0$, such that for all $t \ge 0$, $d'_i(t) \ge \zeta^{(i)}_{d^*}(t)$ for all $i \in \mathcal{U}_K$, and $d'_0(t) \le \zeta^{(0)}_{d^*}(t)$.

%Note that $d'_i(\cdot)\overset{d}{=}\zeta_{d'_i(0)}(\cdot)\overset{d}{\geq} \zeta_m(\cdot)$ for $1\leq i\leq K$ and, on the event $\AA_K$, $d'_0(\cdot)\overset{d}{=}\zeta_m(\cdot)$.

%It follows from the construction of the collapsed branching process embedding that $\{d'_i(t)\}_{0\leq i\leq K}$ are independent. So we can construct a collection of i.i.d. processes $\{\zeta_m^{(i)}(t):t\geq 0\}_{0\leq i\leq K}$, each having the same law as $\zeta_m(\cdot)$, which are naturally coupled with the processes $\{d'_i(t):t\geq 0\}_{0\leq i\leq K}$ such that, for $0\leq i\leq K$, $d'_i(\cdot) \ge \zeta_m^{(i)}(\cdot)$. Moreover, on the event $\AA_K$, $d'_0(\cdot)=\zeta_m^{(0)}(\cdot)$.

By Lemma \ref{mprop}, for $i \in \mathcal{U}_K \cup \{0\}$, the process 
$$M^{(i)}(t):=\sum_{i=d^*}^{\zeta_{d^*}^{(i)}(t)-1}\frac{1}{f(i)}-t \ \overset{d}{=}\sum_{i=1}^{\xi_{d^*-1}(t)}\frac{1}{f_{d^*-1}(i)}-t, \ t \ge 0,$$
(defined conditional on $\FF_K$, on the event $\AA_{\beta, d^*, K}$) is a martingale which converges almost surely to a random limit $M^{(i)}(\infty)$ whose law is absolutely continuous with respect to the Lebesgue measure.

As $\{M_i(\infty) : i \in \mathcal{U}_K \cup \{0\}\}$ are i.i.d. random variables with absolute continuous laws with respect to the Lebesgue measure, for any $i \in \mathcal{U}_K\cup \{0\}$ and $1 \le j \le |\mathcal{U}_K| + 1$, the probability of $M^{(i)}(\infty)$ being the $j$-th largest in $\{M^{(l)}(\infty)\}_{l\in \mathcal{U}_K \cup \{0\}}$ is exactly $1/(|\mathcal{U}_K|+1)$. Therefore, on the event $\AA_{\beta, d^*, K}$,
\begin{align*}
&\liminf_{n\to\infty}P(v_0\notin H_{\lfloor \beta K/2\rfloor,\Psi}(\GG_n)|\FF_K)\\
\ge & \liminf_{n\to\infty} P(d'_0(\tau_n-\tau_K) \text{ is not among the $\lfloor \beta K/2\rfloor$ largest values in $\{d'_i(\tau_n-\tau_K)\}_{i \in \mathcal{U}_K \cup \{0\}}$})\\
\geq & \liminf_{n\to\infty} P(\zeta_{d^*}^{(0)}(\tau_n-\tau_K) \text{ is not among the $\lfloor \beta K/2\rfloor$ largest values in $\{\zeta_{d^*}^{(i)}(\tau_n-\tau_K)\}_{i \in \mathcal{U}_K \cup \{0\}}$})\\
=& \liminf_{n\to\infty} P(M^{(0)}(\tau_n-\tau_K) \text{ is not among the $\lfloor \beta K/2 \rfloor$ largest values in $\{M^{(i)}(\tau_n-\tau_K)\}_{i \in \mathcal{U}_K \cup \{0\}}$})\\
\geq & P(M^{(0)}(\infty) \text{ is not among the $\lfloor \beta K/2\rfloor$ largest values in $\{M^{(i)}(\infty)\}_{i \in \mathcal{U}_K \cup \{0\}}$})\\
=&\sum_{j=\lfloor \beta K/2\rfloor+1}^{|\mathcal{U}_K|+1} \frac{1}{|\mathcal{U}_K|+1}\geq  \frac{1}{2}.
\end{align*}
Taking expectations,
\begin{align*}
&\liminf_{n\to\infty} P(v_0\notin H_{\lfloor \beta K/2 \rfloor,\Psi}(\GG_n))\\
=&\liminf_{n\to\infty} E\left[P(v_0\notin H_{\lfloor \beta K/2 \rfloor,\Psi}(\GG_n)|\FF_K)1_{\AA_{\beta, d^*, K}}\right]\geq \frac{1}{2}P(\AA_{\beta, d^*, K}).
\end{align*}
\end{proof}

\section{Linear attachment: Proof of Theorem \ref{linearattach}}\label{linsec}

%We assume $f(i)=i+\beta$ for some $\beta\geq 0$ and set $m_i\equiv m$ for some $m\in \mathbb{N}$. 

%\begin{lemma}
%Under the assumptions of $f$ in Theorem \ref{linearattach}, there exists some constant $C'_1>0$ so that 
%$$K_\Psi(\ep)\geq \frac{C'_1}{\ep^{\frac{2m+\beta}{m(m+\beta)}}}.$$
%\end{lemma}

\subsection{Lower bound}

We will use Lemma \ref{lbpart1} with $\beta=1$ and $d^* = m$. In this case, $\AA_{\beta,d^*,K}$ simplifies to
$$
\AA_K = \{v_0 \text{ has degree $m$ in } \GG_K\}.
$$
A lower bound for $P(\AA_K)$ is obtained in the following lemma.

\begin{lemma}\label{lbpart2}
There exists $C>0$ depending on $f$ such that for all $K\geq 1$, $P(\AA_K)\geq C/K^{\frac{m( m+\beta)}{2 m+\beta}}$.
\end{lemma}
\begin{proof}
Let $i_0\geq 1$ be such that $\frac{m+ \beta}{(2 m+\beta)i_0+\beta}\leq 1/2$ and write $C_0=P(\AA_{i_0})$. Then, for $K > i_0$, $\AA_K = \AA_{i_0} \cup \{\text{no edge from $v_{i_0+1},\dots,v_K$ attaches to $v_0$}\}$. Hence,
\begin{align*}
P(\AA_K)&\geq P(\AA_{i_0})\prod_{i=i_0}^{K-1} \prod_{im <k\leq (i+1)m} \left( 1-\frac{f(m)}{\sum_{j=0}^{i} f(d_j(k-1))}\right)\\
&\geq P(\AA_{i_0})\prod_{i=i_0}^{K-1} \left( 1-\frac{f(m)}{\sum_{j=0}^{i} f(d_j(s_i))} \right)^m\\
&= C_0\prod_{i=i_0}^{K-1} \left( 1-\frac{ m+\beta}{(2 m+\beta)i+\beta} \right)^m\\
&=C_0\exp\left(m \sum_{i=i_0}^{K-1} \log \left( 1-\frac{ m+\beta}{(2 m+\beta)i+\beta} \right)\right).
%&=C_0\exp\left( \sum_{i=i_0}^{K-1} \log \left( 1-\frac{\lambda m+\beta}{(2\lambda m+\beta)i+\beta} \right)\right)\\
\end{align*}
Using the observation that $\log(1-x) \ge -x - x^2$ for $x \in (0,1/2],$ we have 
$$P(\AA_K)\geq C\exp \left(-\frac{m( m+\beta)}{2m+\beta} \sum_{i=i_0}^{K-1}\frac{1}{i+\beta/(2 m+\beta)}\right) \geq C'\exp\left(-\frac{m( m+\beta) }{2 m+\beta}\log K\right),$$
for positive constants $C, C'$ not depending on $K$.
\end{proof}

\begin{proof}[Proof of lower bound in Theorem \ref{linearattach}]
Lemma \ref{lbpart1} and Lemma \ref{lbpart2} together yield
$$P(v_0\notin H_{\lfloor K/2\rfloor,\Psi}(\GG_n))\geq C/K^{\frac{m( m+\beta)}{2 m+\beta}}$$
for some $C>0$, which then leads to the budget lower bound 
$$K_{\Psi}(\ep)\geq  \left(\frac{C}{\ep}\right)^{\frac{2m+\beta}{m(m+\beta)}}.$$
\end{proof}

\subsection{Upper bound}\label{linearub}

To obtain the budget upper bound, we will show that, with high probability, the degree of the root grows at a certain rate (see event $A_n$ below). On this event, the degree of the root has enough of a head start so that the degree of the $(n+1)$-th added vertex (for large $n$) never crosses that of the root at any future time with high probability. Precise quantitative bounds on the associated probabilities is key to obtaining the budget upper bound.

Throughout this section $m_i \equiv m \ge 1$. All constants that appear below can depend on $m, \beta$, but not $n$.

Let $\{\alpha_n\}$ be an increasing sequence that goes to infinity as $n$ goes to infinity, which will be specified later. Recall that $s_n=\sum_{i=1}^n m_i=mn$. For $i \ge 1$ and $l > s_{i-1}$, recall that $d_i(l)$ denotes the degree of $v_i$ in $\GG^*_l$ (the random graph, index by attached edges, at time $l$). Define the following events for $n\geq 1$,
\beq\label{A_n}
A_n=\{ \Phi_1(d_0(s_n)) \geq \alpha_n\}
\eeq
and 
\beq \label{E_n}
E_n=\{ d_n(k)\geq d_0(k) \text{ for some }k\geq s_n\}.
\eeq

Our goal is to find some $n_0(\ep)\in\mathbb{N}$ so that for all $n\geq n_0(\ep)$, $P(E_n)\leq \ep$. This leads to an upper bound on the budget 
$K_{\Psi}(\ep)\leq n_0(\ep)$. This will be done in two steps. Step 1 estimates $P(E_n\cap A_n)$ and Step 2 estimates $P(A_n^c)$.

\mn
\textbf{Step 1:} The first step is to obtain an upper bound on $P(E_n\cap A_n)$ following the same approach as in the proof of Theorem 4.2 in  \cite{BBpersistence}.
To be self-contained we will present the calculation. The expressions are intentionally left in terms of general $f$ till \eqref{bn} as they are also used in the proof of Theorem \ref{general}.
 
For $A\in \mathbb{N}_0$, recall that $\xi_A(\cdot)$ denotes the point process obtained with attachment function $f_A(\cdot)=f(A+\cdot)$. Define $\zeta^{(n)}_1(t)=d_0(n)+\xi_{d_0(s_n)-1}(t)$ and $\zeta^{(n)}_2(t)=m+\xi_{m-1}(t)$, and the corresponding martingales 
\begin{align*}
M^{(n)}_1(t)&=\sum_{k=d_0(s_n)}^{\zeta^{(n)}_1(t)-1}\frac{1}{f(k)}-t=\sum_{k=1}^{\xi_{d_0(s_n)-1}(t)}\frac{1}{f_{d_0(s_n)-1}(k)}-t,\\
M^{(n)}_2(t)&=\sum_{k=m}^{\zeta^{(n)}_2(t)-1} \frac{1}{f(k)}-t=\sum_{k=1}^{\xi_{m-1}(t)}\frac{1}{f_{m-1}(k)}-t.
\end{align*}
Recall $\Phi_1(\cdot)$ from \eqref{phd}. Then 
\begin{align*}
&P(E_n\cap A_n)\\
\leq &P(M^{(n)}_2(t)+\Phi_1(m)\geq M^{(n)}_1(t)+\Phi_1(d_0(s_n)) \text{ for some }t\geq 0, \hh \Phi_1(d_0(s_n))\geq \alpha_n)\\
\leq &P\left( \inf_{s<\infty} M_{d_0(s_n)}(s)\leq -\frac{\Phi_1(d_0(s_n))-\Phi_1(m)}{2},  \Phi_1(d_0(s_n))\geq \alpha_n\right)\\
&+P\left( \sup_{s<\infty} M_m(s)\geq\frac{\Phi_1(d_0(s_n))-\Phi_1(m)}{2}, \Phi_1(d_0(s_n))\geq \alpha_n\right).
%\leq &P\left(\inf_{s<\infty} M_{d_0(s_n)}(s)\leq -\alpha_n/2,  \Phi_1(d_0(n))\geq \alpha_n)+P(\sup_{s<\infty} M_m(s)\geq \alpha_n/2, \Phi_1(d_0(n))\geq \alpha_n)
\end{align*}
Choosing $n$ sufficiently large such that $\frac{\alpha_n-\Phi_1(m)}{2}\geq \frac{\alpha_n}{4}$, we have 
\begin{align*}
P(E_n\cap A_n)&\leq P\left(\inf_{s<\infty} M_{d_0(s_n)}(s)\leq -\frac{\alpha_n}{4},  \Phi_1(d_0(s_n))\geq \alpha_n\right)\\
&+P\left( \sup_{s<\infty} M_m(s)\geq \frac{\alpha_n}{4}, \Phi_1(d_0(s_n))\geq \alpha_n\right).
\end{align*}
Assume $n_0$ is large enough so that $\frac{\alpha_n-\Phi_1(m)}{2}\geq \frac{\alpha_n}{4}$ for all $n \ge n_0$, and $\frac{\alpha_{n_0}}{4}\geq \max\{x_0,x''_0\}$, where $x_0,x''_0$ are the constants appearing in Lemma \ref{mprop}. It follows from the lemma that there exist constants $C,C'>0$ such that for all $n\geq n_0$,
\begin{align}\label{an}
\nonumber &P\left(\inf_{s<\infty} M_{d_0(s_n)}(s)\leq -\frac{\alpha_n}{4},  \Phi_1(d_0(s_n))\geq \alpha_n\right)\\
\nonumber \leq & E \left[ \left(C\exp(-C'(\alpha_n^2))+C\exp\left( -C'\sqrt{f_*(d_0(s_n))}\alpha_n\right)\right)1_{\{\Phi_1(d_0(s_n))\geq \alpha_n\}}  \right]\\
\leq &C\exp(-C'\alpha_n^2) +C\exp\left( -C'\sqrt{f_*(\lfloor \Phi_1^{-1}(\alpha_n)\rfloor)}\alpha_n\right) =: a_n,
\end{align}
and 
\begin{align}\label{bn}
P\left( \sup_{s<\infty} M_m(s)\geq \frac{\alpha_n}{4}, \Phi_1(d_0(s_n))\geq \alpha_n\right)\leq C\exp(-C'\alpha_n^2) =: b_n.
\end{align}
Since $f(i)=i+\beta$, $f_*(\lfloor \Phi_1^{-1}(\alpha_n)\rfloor)= \lfloor \Phi_1^{-1}(\alpha_n)\rfloor + 1 + \beta \ge C\exp(\alpha_n)$. Hence, setting $\alpha_n=M(\log n)^{1/2}, \, n \ge 1,$ for a fixed large $M$, we obtain a constant $C_1>0$ such that 
$$P(E_n\cap A_n)\leq a_n+b_n\leq C_1 n^{-m}, \ n \ge 1.$$

\mn
\textbf{Step 2:} Next we will estimate $P(A_n^c)$. 

%For intuition we first consider the case where $m=1$ and $s_n=n$. Let $\xi$ be the point process defined in \eqref{pprocess} that represents the degree of $v_0$, and recall that $T_n$ denotes the arrival time of $v_n$. 

%It is easy to see 
%\begin{align*}
%P(\Phi_1( d_0(s_n))<\alpha_n)&= P(\Phi_1(\xi(T_{s_n}))<\alpha_n)\\
%&\leq P\left(T_{s_n}-T_1<\sigma_n \right)+P\left(\bar \xi( \sigma_n)\leq \Phi_1^{-1}(\alpha_n)\right)\\
%&\leq P\left(T_{s_n}<\sigma_n \right)+P\left(\bar \xi( \sigma_n)\leq e^{2\alpha_n}\right),
%\end{align*}
Recall that $\{\GG^*_{k}: k\geq s_1\}$ denotes the edge indexed random graph sequence. $\{\GG_n: n\geq 1\}$ has the same law as $\{\GG_{s_n}^* :n\geq 1\}$. Recall $f(i) = i+\beta$ for $i \ge 1$.
%constructed from the attachment function $f(i)=i+\beta$ where each incoming vertex has $m$ edges. Recall that $s_n=mn$ and $d_i(l)$ denotes the degree of $v_i$ in $\GG^*_l$.

We will construct another sequence of graphs $\{{\bar\GG}_k: k\geq s_1\}$ by attaching one vertex with one edge at each step using attachment function $\bar f(i)=i+\beta/m$.  The construction starts with a fixed graph $\bar\GG_{s_1}$ with vertices $\{\bar{v}_1,\dots, \bar{v}_{2m}\}$ and an edge between $\bar v_i$ and $\bar v_{i+m}$ for each $1\leq i \leq m$. 

Observe that $\bar\GG_{k}$ has vertices $\{ \bar v_i: 1\leq i\leq k+ m\}$. For $k \ge s_1$, $1\leq i\leq k+ m$, let $\bar d_i(k)$ denote the degree of vertex $\bar v_i$ in $\bar \GG_k$. Define $d''_0(k)=\sum_{i=1}^{m}\bar d_i(k)$.  Note that $d''_0(s_1)=d_0(s_1)=m$. For any  $k\geq s_1$,
$$P(d''_0(k+1)=d''_0(k)+1| \bar{d}_i(j), 1\leq i\leq m, j\leq k)=\frac{ \sum_{i=1}^{m} \bar f(\bar d_i(k))}{\sum_{j=1}^{k+m} \bar f(\bar d_j(k))}.$$

\begin{lemma}\label{coupling0}
There exists a coupling of the random graph sequences $\{\GG^*_{k}: k \geq s_1\}$ and $\{{\bar\GG}_k: k\geq s_1\}$ such that $d''_0(k)\leq d_0(k)$ for all $k\geq s_1$.
\end{lemma}
\begin{proof}
We proceed by induction. Recall that $d''_0(s_1)=d_0(s_1)=m$. Suppose for some $n\geq 1$ and some $s_{n}\leq k <s_{n+1}$, we have coupled the random graph sequences $\{\GG^*_{l}:  s_1 \le l \le k\}$ and $\{{\bar\GG}_k: s_1 \leq l \le k\}$ such that $d''_0(i)\leq d_0(i)$ for $i\leq k$. We want to extend the coupling to $s_1 \le l \le k+1$ such that $d_0''(k+1) \le d_0(k+1)$. For $ s_{n}\leq k<s_{n+1}$, in the construction of the original $\{\GG^*_j\}$,
$$P(\text{other end of $e_{k+1}$ attached to $v_0$} \, | \, \GG_j^*, j\leq k) =\frac{f(d_0(k))}{\sum_{j=0}^{n}f(d_j(k))}.$$
Note that, for $s_{n}\leq k< s_{n+1}$,
$$
\sum_{j=1}^{k+m} \bar f(\bar d_j(k)) = 2k + \frac{(k+m)\beta}{m} \geq 2k + (n+1)\beta = \sum_{j=0}^{n} f(d_j(k)).
$$
From this relation and the induction assumption, for $s_n\leq k<s_{n+1}$,
\begin{align*}
\frac{f(d_0(k)}{\sum_{j=0}^n f(d_j(k))} \ge \frac{f(d''_0(k))}{\sum_{j=1}^{k+m} \bar f(\bar d_j(k))}=\frac{\sum_{i=1}^{m} \bar f(\bar d_i(k))}{\sum_{j=1}^{k+m} \bar f(\bar d_j(k))},
\end{align*}
that is,
\begin{multline*}
P(\text{other end of $e_{k+1}$ attached to $v_0$} \, | \, \GG_j^*, j\leq k)\\
\geq P(\text{other end of $e_{k+1}$ attached to one of $\bar{v}_1,\dots, \bar{v}_m$} \, | \, \bar{\GG}_j, j\leq k).
\end{multline*}
This can be used to extend the coupling to $\{\GG^*_{l}:  s_1 \le l \le k+1\}$ and $\{{\bar\GG}_l: s_1 \leq l \le k+1\}$ such that $d''_0(i)\leq d_0(i)$ for $i\leq k+1$. Using the Kolmogorov extension theorem, this completes the proof.
\end{proof}

From Lemma \ref{coupling0} it follows that, for $n\geq 1$,
$$
P(A_n^c)= P(\Phi_1(d_0(s_n)) < \alpha_n)\leq P( \Phi_1(d''_0(s_n))<\alpha_n)\leq P(\cap_{i=1}^{m} \{ \Phi_1(\bar d_i(s_n))<\alpha_n\}),
$$
where recall $\Phi_1(l)=\sum_{i=1}^{l-1}\frac{1}{f(i)}$.
Note that, for any $1 \le i \le m$, $k \ge s_1$, the `offspring' subtree of $\bar v_i$ (subtree corresponding to the connected component of $\bar{\GG}_k$ containing $\bar v_{i}$, rooted at $\bar v_i$) can be described by an embedding into the continuous time branching process (analogous to Lemma \ref{ctbpem})
$$\widetilde{BP}^{(i)}_{\bar f}(t):= BP^{(i)}_{\bar f}(t)\cup BP^{(i+m)}_{\bar f}(t),$$
where $\{BP^{(i)}_{\bar f}(t): 1\leq i\leq 2m\}$ are independent branching processes as in Definition \ref{CTBP} constructed using the attachment function $\bar f$. Hence, for any $n \ge 1$, the degrees $\{\bar d_i(s_n): 1\leq i\leq m\}$ are independent and have the same distribution, which leads to
 \beq\label{error0}
 P(A_n^c)\leq  P(\Phi_1(\bar d_1(s_n)<\alpha_n)^m.
 \eeq

Similarly as in \eqref{pprocess}, we define $\bar \xi$ to be the point process under the attachment function $\bar f$, which is a Yule process with birth rate 1 and immigration rate $\beta/m$. For $k > s_1$, let $\bar T_k := \inf\{t \ge 0: \sum_{i=1}^{2m} |BP^{(i)}_{\bar f}(t)| = k\}$, and set $\bar T_{s_1}=0$. Let $\bar \sigma_n= \frac{m\log n}{2 m+\beta}-\gamma_n$, where $\{\gamma_n\}_{n\geq 1}$ is an increasing sequence diverging to infinity, such that $\gamma_n/\log n \rightarrow 0$ as $n \rightarrow \infty$, that will be chosen later. There exists $c >0$ such that, for all large $l$,
$\Phi_1(l)=\sum_{i=1}^{l-1}\frac{1}{f(i)}=\sum_{i=1}^{l-1}\frac{1}{i+\beta}\geq \log l - \log c,$
which implies $\Phi^{-1}_1(l)\leq ce^{l}$ for large $l$.
Then, for sufficiently large $n$, 
\begin{align*}
P(\Phi_1(\bar d_1(s_n))<\alpha_n)&= P(\Phi_1( \bar \xi(\bar T_{s_n}))<\alpha_n)\\
&\leq P\left(\bar T_{s_n}< \bar \sigma_n \right)+P\left(\bar \xi(\bar \sigma_n)\leq \Phi_1^{-1}(\alpha_n)\right)\\
&\leq P\left(\bar T_{s_n}< \bar \sigma_n \right)+P\left(\bar \xi(\bar \sigma_n)\leq c e^{\alpha_n}\right).
\end{align*}

It follows from \eqref{yulemgf} that for any $\theta > 0, n \in \mathbb{N}$,
\begin{align*}
P\left(\bar \xi(\bar \sigma_n)\leq c e^{\alpha_n}\right)
\nonumber &=P\left(\exp(-\theta \bar \xi(\bar \sigma_n))\geq \exp \left\{-\theta c e^{\alpha_n}\right\}\right)\leq \exp \left\{\theta c e^{\alpha_n}\right\}M(-\theta, \bar \sigma_n)\\
\nonumber&=\exp \left\{\theta c e^{\alpha_n}\right\}\cdot \frac{e^{-\theta}}{((1-e^{-\theta})e^{\bar \sigma_n}+e^{-\theta})^{1+\beta/m}}.
%\nonumber &\leq \frac{C \exp \left\{\theta c e^{\alpha_n}\right\}}{\exp(\bar \sigma_n (1+\beta/m))}\\
%&=Cn^{-\frac{ m+\beta}{2 m+\beta}}\exp\left\{\theta c e^{\alpha_n}+(1+ \beta/m)\gamma_n\right\},
\end{align*}
Hence, choosing $\theta = e^{-\alpha_n}$, we obtain for any $n \in \mathbb{N}$,
\begin{align}\label{error1}
P\left(\bar \xi(\bar \sigma_n)\leq c e^{\alpha_n}\right) \le \frac{e^{c}}{((1-e^{-e^{-\alpha_n}})e^{\bar \sigma_n}+e^{-e^{-\alpha_n}})^{1+\beta/m}}\nonumber\\
 \le Ce^{-(\bar \sigma_n - \alpha_n)(1 + \beta/m)} = C n^{-\frac{m+\beta}{2m + \beta}}e^{(\alpha_n + \gamma_n)(1 + \beta/m)}.
\end{align}
%for some positive constant $C$ depending on $m, \beta$, but not $n$.

It remains to estimate $P(\bar T_{s_n}< \bar \sigma_n)$. Let $\{E_k: k\geq 2m\}$ be a set of i.i.d. exponential random variables with intensity 1. Note that, for any $l \ge m+1$, 
\begin{align*}
\bar T_l& \overset{d}{=} \sum_{k=m}^{l-1}\frac{E_k}{ \sum_{j=1}^{k+m}\bar f(\bar d_j(k))}=\sum_{k=m}^{l-1}\frac{E_k}{ \sum_{j=1}^{k+m}(\bar d_j(k)+\beta/m)}\\
&=\sum_{k=m}^{l-1}\frac{E_k}{2k+(k+m)\beta/m}=\sum_{k=m}^{l-1}\frac{E_k}{(2+\beta/m)k+\beta}.
\end{align*}

Take $\theta_n=\log n$. Then, for sufficiently large $n$, 
\begin{align*}
P(\bar T_{s_n}<\bar \sigma_n)&=P(e^{-\theta_n \bar T_{s_n}}\geq e^{-\theta_n \bar \sigma_n})\leq e^{\theta_n \bar \sigma_n} Ee^{-\theta_n \bar T_{s_n}}\\
&= e^{\theta_n \bar \sigma_n} \exp\left( -\sum_{k=m}^{s_n-1} \log \left(1+\frac{\theta_n}{(2+\beta/m)k+\beta}\right) \right)\\
&\leq Ce^{\theta_n \bar \sigma_n} \exp\left( -\sum_{k=\lfloor \theta_n \rfloor + 1}^{s_n-1} \frac{\theta_n}{(2+\beta/m)k+\beta}\right)\quad \text{(by Taylor expansion)}\\
&\leq Ce^{\theta_n \bar \sigma_n}\exp\left ( -\frac{\theta_n m}{2m+\beta} (\log n - \log \log n - C') \right)\\
&=C e^{-\theta_n (\gamma_n- \frac{m}{2m + \beta}\log \log n - \frac{ C' m }{2m+\beta})}.
\end{align*}
Take $\gamma_n=\frac{2m \log\log n}{2m+\beta}$. Then, for any $M'>0$,
\beq \label{error2}
P(\bar T_{s_n}< \bar \sigma_n)\leq n^{-M'}
\eeq
when $n$ is sufficiently large. Collecting the terms \eqref{error0}, \eqref{error1} and \eqref{error2} we have for any $M'>0$, $\theta \ge 1$,
\begin{align*}
P(A_n^c)&\leq  \left(P\left(\bar T_{s_n}<\bar \sigma_n \right)+P\left(\bar \xi(\bar \sigma_n)\leq ce^{\alpha_n}\right)\right)^m\\
&\leq \left(Cn^{-\frac{ m+\beta}{2 m+\beta}}\exp\left((1+ \beta/m)(\alpha_n+\gamma_n)\right) +n^{-M'}\right)^m\\
&\leq C' n^{-\frac{m(m+\beta)}{2m+\beta}} e^{(\alpha_n + \gamma_n)(m + \beta)} + C' n^{-M'm}.
\end{align*}
Take $M' = 1$. Recall from Step 1 that $\alpha_n = M(\log n)^{1/2}$ for some fixed $M$, and recall $\gamma_n=\frac{2m \log\log n}{2m+\beta}$. Then,
$$P(A_n^c)\leq C'' n^{-\frac{m(m+\beta)}{2m+\beta}} \exp\left(C(\log n)^{1/2}\right), \ n \ge 1.$$
This completes Step 2. 

\begin{proof}[Proof of upper bound in Theorem \ref{linearattach}]
Combining the error probabilities obtained in Step 1 and Step 2, we have 
\begin{align*}
P(E_n)&\leq P(E_n\cap A_n)+P(A^c_n)\leq c_1 n^{-m} + C''n^{-\frac{m(m+\beta)}{2m+\beta}} \exp\left(C(\log n)^{1/2}\right)\\
&\leq (c_1 + C'')n^{-\frac{m(m+\beta)}{2m+\beta}} \exp\left(C(\log n)^{1/2}\right).
\end{align*}
This leads to the upper bound on $K_\Psi(\cdot)$:
$$K_\Psi(\ep)\leq \frac{C_1}{\ep^{\frac{2m+\beta}{m(m+\beta)}}}\exp\left(\sqrt{C_2\log \frac{1}{\ep}}\right), \ \ep \in (0,1),$$
for positive constants $C_1,C_2$ depending on $m,\beta$ but not $\ep$.
\end{proof}

\section{General attachment: Proof of Theorem \ref{general}}\label{gensec}
In this section, we prove Theorem \ref{general} in the course of the following lemmas.
\subsection{When $m=1$}
\begin{lemma}\label{lbm1}
Under the assumptions in Theorem \ref{general} with $m=1$, there exists some $C_1>0$ so that
$$ K_\Psi(\ep)\geq C_1 \ep^{-\frac{\lambda^*}{f_*}}, \quad \text{ for all } \ep \in (0,1).$$
\end{lemma}
\begin{proof}
Let $i_0$ be the minimum index such that $f(i_0)=f_*$. Note that such $i_0$ always exists because $f(k) \rightarrow \infty$ as $k \rightarrow \infty$ by Assumption (A1). We will apply Lemma \ref{lbpart1} with $d^* = i_0$ and an appropriate choice of $\beta$ to be fixed later. For $K\geq 1$, $\beta \in (0,1]$, recall the event
\begin{align*}
\AA_{\beta, i_0, K} &=\{ v_0 \text{ has degree $\leq i_0$ in $\GG_{K}$ and there are at least $\lfloor \beta K\rfloor$}\\
& \qquad \qquad \text{ other vertices with degree $\ge i_0$ in $\GG_{K}$}\}.
\end{align*}
First we give a lower bound on $P(\AA_{\beta, i_0, K})$ for an appropriate choice of $\beta$. Recall the embedding of our discrete tree process in the continuous time branching process described in Lemma \ref{ctbpem}.
Recall the branching process $BP^{(1)}_f(\cdot)$ representing the descendants of $v_1$. Let $\{\xi_{1,i}(\cdot) : i \in \mathbb{N}\}$ denote the point processes of reproduction of the vertices descending from $v_1$. For $\delta \in (0,1)$, let $t_\delta:=\frac{1}{2\lambda^*}\log (1/\delta)$ and $p:=P( \xi(t_\delta) \ge i_0)$ ($\xi(\cdot)$ is defined in \eqref{pprocess}). Define the events
\begin{align*} 
G_{1,K}&:=\left\{ BP^{(1)}_f( \frac{1}{\lambda^*}\log (\delta K))\geq (\delta K)/2\right\},\\
G_{2,K}&:=\left\lbrace \sum_{i=1}^{\lfloor \delta K/2 \rfloor} 1_{\{\xi_{1,i}(t_{\delta}) \ge i_0\}} \ge  p\delta K/3\right\rbrace,\\
G_{3,K}&:= \left\{ \frac{1}{\lambda^*}\log K -t_\delta< T_K< \frac{1}{\lambda^*}\log K +t_\delta \right\}.
\end{align*} 

Then we can observe that on $G_{1,K}\cap G_{3,K}$ at least $(\delta K)/2$ vertices have arrived before time $\frac{1}{\lambda^*}\log (\delta K)$, and hence, the point process of reproduction for each of these vertices has run for at least time $t_{\delta} = \frac{1}{2\lambda^*}\log(1/\delta)$ before time $T_K$. Thus, each of these vertices independently has degree $\ge i_0$ in $\widetilde{BP}_f(T_K)$ with probability at least $p$. Event $G_{2,K}$ then implies that there are at least $\lfloor \delta pK/3\rfloor$ vertices with degree $\ge i_0$ in $\GG_K$. Therefore, 
$$A_{\delta p/3, i_0, K}\supseteq G_{1,K}\cap G_{2,K}\cap G_{3,K} \cap \{\xi^{(0)}(T_K)\leq i_0\},$$
where $\xi^{(0)}(t)$ denotes the degree of $v_0$ at time $t$.
Since $\xi^{(0)}(t)$ and $BP^{(1)}(t)$ are independent, 
\begin{align}\label{badevent}
P(\AA_{\delta p/3, i_0, K})\geq P(\xi^{(0)}(T_K)\leq i_0| G_{3,K})\cdot P(G_{1,K}\cap G_{2,K}\cap G_{3,K}).
\end{align}
The first term in the lower bound above satisfies 
\begin{align*}
P(\xi^{(0)}(T_K)\leq i_0| G_{3,K}) &\geq P(\{\xi^{(0)}(T_K)\leq i_0\}\cap G_{3,K})\\
&\geq E\left[ P\left(\left\{\xi^{(0)}\left(\frac{1}{\lambda^*}\log K +t_\delta\right)\leq i_0\right\}\cap G_{3,K}|T_K\right)\right]\\
&\geq E\left[ 1_{ \{\frac{1}{\lambda^*}\log K -t_\delta< T_K< \frac{1}{\lambda^*}\log K +t_\delta\}} P\left(\frac{E_{i_0}}{f(i_0)} > \frac{1}{\lambda^*}\log K +t_\delta\right)\right]\\
&=E\left[ e^{-f_*\left(\frac{1}{\lambda^*}\log K +t_\delta\right)} 1_{ \{\frac{1}{\lambda^*}\log K -t_\delta< T_K< \frac{1}{\lambda^*}\log K +t_\delta\}} \right]\\
&= K^{-\frac{f_*}{\lambda^*}} e^{-f_*t_\delta}P(G_{3,K}).
\end{align*}
Next, we will show that $P(G_{3,K})$ and $P(G_{1,K}\cap G_{2,K}\cap G_{3,K})$ are lower bounded by some positive constants uniformly in $K$. We start with 
\begin{align*}
\limsup_{K\to\infty} P(G_{1,K}^c)&=\limsup_{K\to\infty}P( BP^{(1)}_f(\frac{1}{\lambda^*}\log (\delta K)) < (\delta K)/2)\\
&=\limsup_{K\to\infty}P\left(e^{-\lambda^* (\frac{1}{\lambda^*}\log (\delta K))}BP^{(1)}_f(\frac{1}{\lambda^*}\log (\delta K)) < \frac{1}{2}\right)\\
&\le P(W_\infty \le 1/2).
\end{align*}
Next, as $\sum_{i=1}^{\lfloor \delta K/2 \rfloor} 1_{\{\xi_{1,i}(t_{\delta}) \ge i_0\}} \sim Binomial(\lfloor \delta K/2 \rfloor,p)$,  $\lim_{K\to\infty}P(G_{2,K}^c)= 0$ by standard results on Binomial random variables. Lastly, using \eqref{tnlim},
\begin{align*}
\limsup_{K\to\infty}P(G_{3,n}^c)&\leq \limsup_{K\to\infty}P(T_K>\frac{1}{\lambda^*}\log K+t_\delta)+\limsup_{K\to\infty}P(T_K<  \frac{1}{\lambda^*}\log K -t_\delta)\\
&\le P\left(W_\infty \le \exp(-\frac{1}{2}\log(1/\delta))\right)+P\left(W_\infty \ge \exp(\frac{1}{2}\log(1/\delta))\right),
\end{align*}
where both term can be arbitrarily small when $\delta$ is small since $P(W_\infty>0)=1$ and $EW_\infty < \infty$. Choosing and fixing $\delta$ sufficiently small and collecting all the terms gives
\begin{align*}
&\liminf_{K\to\infty}P(G_{1,n}\cap G_{2,K}\cap G_{3,K})\geq 1-\limsup_{K\to\infty}(P(G_{1,K}^c)+P(G_{2,K}^c)+P(G_{3,K}^c))\\
&\ge 1-P(W_\infty \le 1/2) - P\left(W_\infty \le \exp(-\frac{1}{2}\log(1/\delta))\right) - P\left(W_\infty \ge \exp(\frac{1}{2}\log(1/\delta))\right)>0,
\end{align*}
and 
$$
\liminf_{K \to \infty}P(G_{3,K}) \ge \liminf_{K\to\infty}P(G_{1,K}\cap G_{2,K}\cap G_{3,K})>0.
$$
In the above, we have used $P(W_{\infty} \le 1/2) < 1$ as $W_{\infty}$ is supported on all of $\mathbb{R}_+$, by Theorem \ref{brlim}.

It thus follows from \eqref{badevent} that, for this fixed choice of $\delta \in (0,1)$, there exists some constant $C>0$ so that for all $K \geq 1$,
$$P(\AA_{\delta p/3, i_0, K})\geq C K^{-\frac{f_*}{\lambda^*}}.$$
Now, by Lemma \ref{lbpart1} with $d^* = i_0$ and $\beta = \delta p/3$, we obtain 
$$\liminf_{n\to\infty} P(v_0\notin H_{\lfloor\beta K/2\rfloor,\Psi}(\GG_n))\geq \frac{1}{2}P(\AA_{\delta p/3, i_0, K})\geq \frac{C}{2} K^{-\frac{f_*}{\lambda^*}},$$
which leads to the lower bound 
$$ K_\Psi(\ep)\geq \left(\frac{C}{2\ep}\right)^{\frac{\lambda^*}{f_*}}.$$

\end{proof}

\begin{lemma}\label{ubg1}
Under the assumptions in Theorem \ref{general} with $m=1$. For any $\delta \in (0,1)$, we have
$$ K_\Psi(\ep)\leq C''_\delta \ep^{-\frac{\lambda^*}{(1-\delta)f_*}},$$
where $C''_\delta$ is a constant depending on $\delta$.
\end{lemma}
\begin{proof}
Let $ A_n, E_n$ be the events defined in \eqref{A_n}, \eqref{E_n}. The proof follows essentially the same strategy as the proof in Section \ref{linearub}. A slight difference here is we choose $\alpha_n=\eta\log n$ instead of $M(\log n)^{1/2}$ for some small $\eta>0$.
The same calculation that leads to \eqref{an} and \eqref{bn} yields, 
$$P(E_n\cap A_n)\leq C\exp(-C'\alpha_n^2) +C\exp\left( -C'\sqrt{f_*(\lfloor \Phi_1^{-1}(\alpha_n)\rfloor)}\alpha_n\right).$$
Fix any $M>f_*/\lambda^*$. Since $f_*(\lfloor \Phi_1^{-1}(\alpha_n)\rfloor)\to \infty$ as $n\to\infty$, we have 
\begin{equation}\label{etaeq}
P(E_n\cap A_n)\leq \hat C_{\eta}n^{-M}, \ n \ge 1.
\end{equation}
for some constant $\hat C_{\eta}>0$ depending on $\eta$.

The calculations for Step 2, though following the same broad strategy, rely on more general properties of the point process $\xi$ defined in \eqref{pprocess}, which we now present. To simplify notation let $\sigma''_n=\frac{1}{\lambda^*}\log n-\gamma_n$. Recall that $\xi(t)$ is the point process that encodes the reproduction times of an individual in the continuous time embedding of our discrete tree process. We can write
$$P(A_n^c)=P(\Phi_1(d_0(n))<\alpha_n)=P(\Phi_1(\xi(T_n))<\alpha_n)\leq P(T_n< \sigma''_n)+P(\Phi_1(\xi(\sigma''_n))<\alpha_n).$$
First we will estimate $P(\Phi_1(\xi(\sigma''_n))<\alpha_n)$. Recall $S_1(\cdot)$ from \eqref{sdef}. Note that 
\begin{align}\label{S1concentrationa}
\nonumber P(\Phi_1(\xi(\sigma''_n))<\alpha_n)&=P(\xi(\sigma''_n)< \Phi^{-1}_1(\alpha_n)) \le P(S_1(\lceil \Phi^{-1}_1(\alpha_n)\rceil)\geq \sigma''_n)\\
&=P(S_1(\lceil \Phi^{-1}_1(\alpha_n)\rceil)- \Phi_1(\Phi^{-1}_1(\alpha_n))\geq \sigma''_n-\alpha_n)\nonumber\\
&\le P\left(S_1(\lceil \Phi^{-1}_1(\alpha_n)\rceil)- \Phi_1(\lceil \Phi^{-1}_1(\alpha_n)\rceil + 1)\geq \sigma''_n-\alpha_n - \frac{2}{f_*(\lfloor\Phi^{-1}_1(\alpha_n)\rfloor)}\right).
\end{align}
For any $l\in \mathbb{N}$ and $\theta\in (0,f_*)$, we have 
\begin{align*}
\log E[ \exp(\theta(S_1(l)-\Phi_1(l+1)))]&=-\theta \Phi_1(l+1)-\sum_{i=1}^{l-1}\log (1-\frac{\theta}{f(i)})\\
&=\sum_{k=2}^\infty \frac{\theta^k}{k}\Phi_k(l+1) \le \theta^2\Phi_2(l+1) \sum_{k=2}^{\infty} \left(\frac{\theta}{f_*}\right)^{k-2}\nonumber\\
&= \frac{f_*}{f_* - \theta}\theta^2\Phi_2(l+1).
\end{align*}
It follows that for any $x>0$, $l\in\mathbb{N}$ and $\theta\in (0,f_*)$,
\beq\label{S1concentrationb}
P(S_1(l)-\Phi_1(l+1)\geq x)\leq e^{-\theta x}\exp\left[\frac{f_*}{f_* - \theta}\theta^2\Phi_2(l+1)\right]\leq  e^{-\theta x}\exp\left[\frac{f_*}{f_* - \theta}\theta^2\Phi_2(\infty)\right].
\eeq
Hence, by \eqref{S1concentrationa} and \eqref{S1concentrationb}, we obtain for any $\theta\in (0,f_*)$,
$$P(\Phi_1(\xi(\sigma''_n))<\alpha_n)\leq e^{2\theta/f_*}e^{-\theta (\sigma''_n-\alpha_n)}\exp\left[\frac{f_*}{f_* - \theta}\theta^2\Phi_2(\infty)\right] = C^*_{\theta} e^{\theta(\alpha_n+\gamma_n)} n^{-\theta /\lambda^*},$$
for any $n \ge 1$, where $C^*_{\theta} := e^{2\theta/f_*}\exp\left[\frac{f_*}{f_* - \theta}\theta^2\Phi_2(\infty)\right]$.
Recall that $\alpha_n=\eta \log n$. Since $\theta$ can be chosen arbitrarily close to $f_*$, for any $\delta>0$ we can choose $\theta=(1-\delta/4)f_*$ so that, writing $C'_{\delta} := C^*_{(1-\delta/4)f_*}$,
$$P(\Phi_1(\xi(\sigma''_n))<\alpha_n)\leq C'_{\delta} e^{f_*(\alpha_n+\gamma_n)} n^{-\frac{(1-\delta/4)f_* }{\lambda^*}}=C'_{\delta}e^{f_*\gamma_n}n^{-\frac{(1-\lambda^*\eta-\delta/4)f_*}{\lambda^*}}, \ n \ge 1.$$

Now, choosing $\gamma_n=\frac{\delta}{4\lambda^*} \log n$ and any integer $k=k_\delta > 4(1-\delta)f_*/(\delta \lambda^*)$, \eqref{Tlower} yields
$$P(T_n< \sigma''_n)\leq C_{k_\delta} n^{-\delta  f_* k_\delta/4}\leq C_{k_\delta} n^{-\frac{(1-\delta)f_*}{\lambda^*}}.$$
Therefore, taking $\eta = \delta/(2\lambda^*)$,
\begin{align*}
P(A_n^c)&\leq C'_{\delta}e^{f_*\gamma_n}n^{-\frac{(1-\lambda^*\eta-\delta/4)f_*}{\lambda^*}}+C_{k_\delta} n^{-\frac{(1-\delta)f_*}{\lambda^*}} = \tilde C_\delta  n^{-\frac{(1-\delta)f_*}{\lambda^*}},
\end{align*} 
where $\tilde C_\delta := C'_{\delta} + C_{k_\delta}$. Therefore, using \eqref{etaeq} and recalling $M>f_*/\lambda^*$, we obtain for all $n\geq 1$,
$$P(E_n)\leq P(E_n\cap A_n)+P(A_n^c)\leq \hat{C}'_{\delta} n^{-M}+ \tilde C_\delta  n^{-\frac{(1-\delta)f_*}{\lambda^*}}\leq \tilde{C}'_\delta  n^{-\frac{(1-\delta)f_*}{\lambda^*}},$$
where $\hat{C}'_{\delta} := \hat C_{\delta/(2\lambda^*)}$ and $\tilde{C}'_\delta := \hat{C}'_{\delta} + \tilde C_\delta$.
The above bound gives the following upper bound on $K_\Psi(\ep)$:
$$K_\Psi(\ep)\leq \left\lceil\left(\frac{\tilde{C}'_\delta}{\ep}\right)^{\frac{\lambda^*}{(1-\delta)f_*}}\right\rceil,$$
which proves the lemma.
\end{proof}

\subsection{For general $m>1$.}
\mn
\textbf{Lower bound.} 
\begin{lemma}\label{lbgm}
Under the assumptions in Theorem \ref{general} with $m>1$, there exists some $C_1>0$ so that  
$$K_\Psi(\ep)\geq C_1\ep^{-\frac{f_*}{mf(m)}}.$$
\end{lemma}
\begin{proof}
The proof is essentially the same as that of Lemma \ref{lbpart2} with a different lower bound on $P(\AA_K)$ for the event $\AA_K=\{ v_0 \text{ has degree $m$ in $\GG_K$}\}$. Let $i_0\geq 1$ be such that $\frac{f(m)}{f_*(i_0+1)}\leq 1/2$.
\begin{align*}
P(\AA_K)&=E(1\{\AA_{K-1}\}P(\AA_K|\GG_{K-1}))\\
&= E\left(1\{\AA_{K-1}\}  \prod_{i=1}^m \left(1-\frac{f(m)}{\sum_{j=0}^{K-1}f(d_j(s_{K-1}+i-1))}\right)\right)\\
&\geq E\left(1\{\AA_{K-1}\} \left(1-\frac{f(m)}{Kf_*}\right)^m\right)\geq P(\AA_{i_0})\prod_{i=i_0}^{K-1} \left( 1-\frac{f(m)}{(i+1)f_*} \right)^m\\
%&\geq C_0\prod_{i=i_0}^{K-1} \left( 1-\frac{f(m)}{f_*(i+1)} \right)^m\\
&=P(\AA_{i_0})\exp\left(m \sum_{i=i_0}^{K-1} \log \left( 1-\frac{f(m)}{(i+1)f_*} \right)\right)\geq C K^{-\frac{mf(m)}{f_*}}.
\end{align*}
Again, applying Lemma \ref{lbpart1} with the above estimate gives the desired the result.
\end{proof}

\subsubsection{Upper bound}
\begin{lemma}\label{ubgm}
Under the assumptions in Theorem \ref{general} with $m>1$, for any $\delta \in (0,1)$, there exists $\bar{C}_{\delta}>0$ such that for all $\ep \in (0,1)$,
$$K_\Psi(\ep) \leq \bar{C}_{\delta}\ep^{-\frac{2C_f}{(1-\delta)f_*}}.$$
\end{lemma}
\begin{proof}
We will use the construction and terminology of the collapsed branching process embedding defined in Section \ref{BPpre}. Let $\xi^{(0)}(t)$ denote the `continuous time' degree of the root at time $t$.  Let $\{\sigma^{\circ}_n\}_{n\geq 1}$, $\{\alpha^{\circ}_n\}_{n \ge 1}$ be increasing sequences to be specified later. We follow the same strategy as in Section \ref{linearub}. Define the events $A_n=\{ \Phi_1(d_0(s_n)) \geq \alpha^{\circ}_n\}$
and 
$
E_n=\{ d_n(k)\geq d_0(k) \text{ for some }k\geq s_n\}.
$
We begin by upper bounding $P(A_n^c)$,
\begin{align*}
P(A_n^c)&= P(\Phi_1(d_0(s_n))< \alpha^{\circ}_n)=P(\Phi_1(\xi^{(0)}(\tau_n))<\alpha^{\circ}_n)\\
&\leq P(\Phi_1(\xi^{(0)}(\sigma^{\circ}_n))<\alpha^{\circ}_n)+P(\tau_n<\sigma^{\circ}_n).
\end{align*}
Take any $l \ge 1$. Recall that $\tau_{l+1}=\tau_l+\Theta_{l+1}$, where $\Theta_{l+1}$ is the first time since $\tau_l$ when there are $m$ newly born individuals $v_{l+1,1}, \dots, v_{l+1,m}$. Let $\Theta_{l+1,i}$ denote the arrival time of $v_{l+1,i}$. Also, denote by $\FF_l$ the filtration generated by the collapsed branching process till time $\tau_l$. It is not hard to see that, conditional on $\FF_l$, the distribution of $\Theta_{l+1,1}$ follows the exponential distribution with rate $\sum_{i=0}^l f( d_i(s_l))$. By Assumption (A1), $f(i)\leq C_f i$ for all $i\geq m$, so $\sum_{i=0}^l f( d_i(s_l))\leq C_f\sum_{i=0}^l d_i(s_l)=2s_l C_f$. That is, $\Theta_{l+1,1}$ stochastically dominates an exponential random variable with rate $2s_lC_f$.
Let $\{E_{l,i}: 1\leq i\leq m, l\geq 1\}$ denote a collection of i.i.d. exponential random variables with rate 1. Then 
$$\Theta_{l+1}=\Theta_{l+1,m} \overset{d}{\geq} \sum_{i=1}^m \frac{E_{l,i}}{2s_lC_f}= \sum_{i=1}^m \frac{E_{l,i}}{2mC_f l}.$$
Hence, for $\theta_n=\log n$,
\begin{align*}
E e^{-\theta_n \tau_n}&= E \exp(-\theta_n \sum_{l=2}^{n} \Theta_{l})\leq E \exp\left(-\theta_n\sum_{l=2}^n \sum_{i=1}^m \frac{E_{l,i}}{2mC_f l}\right)\\
&=\exp\left(-m\sum_{l=2}^n \log\left(1+\frac{\theta_n}{2mC_f l}\right)\right)\leq C\exp\left(-m\sum_{l=\lfloor \theta_n/(mC_f)\rfloor+1}^n \frac{\theta_n}{2mC_f l}\right)\\
&\leq C\exp\left( -\frac{\theta_n}{2C_f} (\log n-\log\log n -C')\right),
\end{align*}
where we used $\log(1 + x) \ge x - x^2$ for $x \in [0,1/2]$ to obtain the second inequality above.
Take any $\delta \in (0,1)$. Let $\sigma^{\circ}_n=\frac{(1-\delta/4)\log n}{2C_f}$. Then
\begin{align*}
P(\tau_n<\sigma^{\circ}_n)&\leq e^{\theta_n \sigma^{\circ}_n}E e^{-\theta_n \tau_n}\\
&\leq C e^{\theta_n \sigma^{\circ}_n}\exp\left( -\frac{\theta_n}{2C_f} (\log n-\log\log n -C')\right)\\
&\leq C''n^{-\theta_n\delta/16C_f}\leq \tilde{C}n^{-\frac{f_*}{2C_f}} \ \ n \ge n_{\delta},
\end{align*}
for positive constants $C'', \tilde{C}$, when $n_{\delta}$ is chosen sufficiently large. The same calculation as in \eqref{S1concentrationa} and \eqref{S1concentrationb} implies for any $\theta\in(0,f_*)$,
$$P(\Phi_1(\xi(\sigma^{\circ}_n))<\alpha^{\circ}_n)\leq e^{2\theta/f_*}e^{-\theta (\sigma^{\circ}_n-\alpha^{\circ}_n)}\exp\left[\frac{f_*}{f_* - \theta}\theta^2\Phi_2(\infty)\right] = C^*_{\theta} e^{\theta\alpha^{\circ}_n} n^{-(1-\delta/4)\theta /2C_f},$$
for any $n \ge 1$, where $C^*_{\theta} := e^{2\theta/f_*}\exp\left[\frac{f_*}{f_* - \theta}\theta^2\Phi_2(\infty)\right]$.
Choose $\alpha^{\circ}_n=(\delta/(8C_f)) \log n$. Since $\theta$ can be chosen arbitrarily close to $f_*$, for any $\delta>0$ we can choose $\theta=(1-\delta/4)f_*$ so that, writing $C'_{\delta} := C^*_{(1-\delta/4)f_*}$, 
$$P(\Phi_1(\xi(\sigma^{\circ}_n))<\alpha^{\circ}_n)\leq C'_{\delta} n^{\frac{\delta f_*}{8C_f}-\frac{(1-\delta/4)^2f_*} {2C_f}}\leq C'_{\delta}n^{-\frac{(1-\delta)f_*}{2C_f}} , \ n \ge 1.$$
Again, using \eqref{etaeq} and recalling $M>f_*/(2C_f)$ gives that for all $n\geq 1$,
$$P(E_n)\leq P(E_n\cap A_n)+P(A_n^c)\leq \hat{C}'_{\delta} n^{-M}+  C'_\delta  n^{-\frac{(1-\delta)f_*}{2C_f}}\leq \tilde{C}'_\delta  n^{-\frac{(1-\delta)f_*}{2C_f}},$$
where $\hat{C}'_{\delta} := \hat C_{\delta/(8C_f)}$ and $\tilde{C}'_\delta := \hat{C}'_{\delta} +  C'_\delta$.
The above bound gives the following upper bound on $K_\Psi(\ep)$:
$$K_\Psi(\ep)\leq \left\lceil\left(\frac{\tilde{C}'_\delta}{\ep}\right)^{\frac{2C_f}{(1-\delta)f_*}}\right\rceil,$$
which proves the lemma.
\end{proof}

\begin{proof}[Proof of Theorem \ref{general}]
The lower and upper bounds in part (i) ($m_i \equiv m= 1$ case) follow respectively from Lemma \ref{lbm1} and Lemma \ref{ubg1}. Those for part (ii) ($m_i \equiv m > 1$ case) respectively follow from Lemma \ref{lbgm} and Lemma \ref{ubgm}.
\end{proof}

\section{Non-persistent regime: Proof of Theorem \ref{nonpersistent}}\label{nonsec}

This section is dedicated to the proof of Theorem \ref{nonpersistent}. We will assume throughout that $m_i \equiv m=1$. For $n \in \mathbb{N}$, denote by $v_{max}(n)$ the youngest vertex with the maximal degree in $\GG_n$.

The following lemma quantifies the probability of $v_{max}(n)$ lying within a certain distance of the root in $\GG_n$.

\begin{lemma}\label{vmaxerr}
Let $r_n=c_1\lambda^*\KK(\frac{1}{\lambda^*}\log n)$ where $c_1$ is defined as in Lemma \ref{height}. There exist $C,C'>0$ so that for all $n\geq 1$,
$$P( v_{max}(n)\notin B_n(v_0,r_n)))\leq C\exp\left(-C'\KK(\frac{1}{\lambda^*}\log n)\right).$$
\end{lemma}
\begin{proof}

Let $a_n=\exp\left( \frac{\lambda^{*2}}{2}\KK(\frac{1}{\lambda^*}\log n)\right)$, $t_n=\frac{1}{\lambda^*}\log a_n$ and $r_n=2c_1t_n$. Recall $\II^*_c(t)$ as defined in \eqref{maxbirth}. Using the continuous time embedding of the discrete tree process,
\begin{align*}
 P(v_{max}(n) \in B_n(v_0,r_n))&\geq P(H( \II^*_c(T_n) ) \leq r_n)\geq P( \II^*_c(T_n) \leq 2t_n, H(2t_n) \leq  r_n).
\end{align*}
Hence, 
\begin{equation}\label{np1}
P(v_{max}(n) \notin B_n(v_0,r_n))\leq P( \II^*_c(T_n) \geq 2t_n)+P( H_{2t_n} \geq  r_n).
\end{equation} 
By Lemma \ref{maxposition} with $\ep=\lambda^*/2$, there exists $C,C'>0$ so that
$$P( \II^*_c(T_n) \geq 2t_n)\leq C\exp(-C'\KK(\frac{1}{\lambda^*}\log n)).$$
Applying Lemma \ref{height} to the second term in the bound \eqref{np1},
$$P( H_{2t_n} \geq  r_n)= P(H_{2t_n}\geq 2c_1t_n)\leq e^{-t_n}\leq \exp\left(-\frac{\lambda^*}{2}\KK(\frac{1}{\lambda^*}\log n)\right).$$
The lemma follows on using the above two bounds in \eqref{np1}.
\end{proof}

The following lemma introduces a key functional $\tilde \alpha^*(\cdot)$ which arises in quantifying the exponential growth rate of the number of vertices within a certain radius of the root (see \eqref{growball}). We will require understanding of $\tilde \alpha^*(\cdot)$ near zero, as described in the next lemma.

\begin{lemma}\label{alphax}
Define $\tilde \alpha^*(x)=\inf_{0\leq \theta\leq \lambda^*} \{ x\log \hat{\mu}(\theta)+\theta\}$. Suppose $f(i)\leq C_0i^{\alpha}$ for some constant $C_0$ and $\alpha\in (0,1/2]$. For any $L>0$, there exists some constant $C>0$ depending on $\alpha, L$ so that 
\begin{equation}\label{al}
\tilde \alpha^*(x)\leq Cx-\frac{2}{1-\alpha}x\log x, \ \ x \in (0,L].
\end{equation}
\end{lemma}
\begin{proof}
All constants in the proof will depend on $\alpha, \lambda^*, f_*$ but not $\theta$. For $x \in (0,1)$, since $\log (1-x)\leq -x$, 
\begin{align*}
\hat{\mu}(\theta)&=\sum_{k=1}^\infty \prod_{i=1}^{k} \frac{f(i)}{\theta+f(i)}= \sum_{k=1}^\infty \exp\left(\sum_{i=1}^{k} \log \left(1-\frac{\theta}{\theta+f(i)}\right)\right)\leq \sum_{k=1}^\infty \exp\left (- \sum_{i=1}^{k} \frac{\theta}{\theta+f(i)} \right).
\end{align*}
Since we are considering $\theta\in [0,\lambda^*]$, $\theta+f(i)\leq \frac{\lambda^*}{f_*}  f(i)+ f(i)=\frac{f_*+\lambda^*}{f_*}f(i)$. It follows that there exists a positive constant $C_1$ such that
\begin{align}\label{lbo}
\hat{\mu}(\theta)&\leq  \sum_{k=1}^\infty  \exp\left(-\frac{\theta f_*}{f_*+\lambda^*} \sum_{i=1}^k \frac{1}{f(i)}\right)\leq \sum_{k=1}^\infty  \exp(-C_1\theta k^{1-\alpha}).
%&\leq \sum_{k=1}^{\theta^{-\frac{2}{1-\alpha}}} \exp(-C_1\theta)+ \sum_{k>\theta^{-\frac{2}{1-\alpha}}}\exp(-C_1\theta k^{1-\alpha})\leq C_1e^{-C\theta} \cdot \frac{1}{\theta^{\frac{1}{1-\alpha}}}
\end{align}
Observe that there exists some constant $C'>0$ so that $$\exp(-C_1\theta k^{1-\alpha})\leq \exp(-C_1k^{(1-\alpha)/2}) \le \frac{C'}{k^2}$$ for all $k> \theta^{-\frac{2}{1-\alpha}}$. Hence,
\begin{align*}
\hat{\mu}(\theta)&\leq \sum_{k=1}^{\lfloor\theta^{-\frac{2}{1-\alpha}}\rfloor} \exp(-C_1\theta)+ \sum_{k>\theta^{-\frac{2}{1-\alpha}}}\exp(-C_1\theta k^{1-\alpha})\\
&\leq \theta^{-\frac{2}{1-\alpha}} \exp(-C_1\theta)+\sum_{k> \theta^{-\frac{2}{1-\alpha}}} \frac{C'}{k^2}\\
&\leq  \theta^{-\frac{2}{1-\alpha}} \exp(-C_1\theta)+C'' \theta^{\frac{2}{1-\alpha}}\leq C_2 \theta^{-\frac{2}{1-\alpha}} \exp(-C_1\theta),
\end{align*}
where $C_2 = 1 + C'' (\lambda^*)^{4/(1-\alpha)}e^{C_1\lambda^*}$. Therefore, 
\begin{align*}
\inf_{0\leq \theta\leq \lambda^*} \{ x\log \hat{\mu}(\theta)+\theta\}&\leq \inf_{0\leq \theta\leq \lambda^*} \bigg\{ x \log \left( C_2e^{-C_1\theta} \cdot \frac{1}{\theta^{\frac{2}{1-\alpha}}}\right)+\theta\bigg\}\\
&= \inf_{0\leq \theta\leq \lambda^*} \bigg\{ x\left(\log C_2-C_1\theta-\frac{2}{1-\alpha}\log \theta\right)+\theta\bigg\}.
\end{align*}
Let $h(\theta)=(1-C_1x)\theta-\frac{2x}{1-\alpha}\log \theta$. The minimum for the function $h$ is achieved at $\hat{\theta}=\frac{2x}{(1-\alpha)(1-C_1x)}$. If $x \in [0, \frac{1}{2C_1} \wedge \frac{(1-\alpha)\lambda^*}{4}]$, $\hat{\theta} \in [0, \lambda^*]$. Hence, for $x \in (0, \frac{1}{2C_1} \wedge \frac{(1-\alpha)\lambda^*}{4} \wedge 1)$,
\begin{align*}
\tilde \alpha^*(x)&\leq x\log C_2+ \frac{2x}{1-\alpha}-\frac{2x}{1-\alpha}\log \left(\frac{2x}{(1-\alpha)(1-C_1x)}\right)\\
&\leq \left(\log C_2+\frac{2}{1-\alpha}(1-\log 2)\right)x-\frac{2}{1-\alpha}\cdot x\log x\\
&= C_3x-\frac{2}{1-\alpha}x\log x,
\end{align*}
where $C_3 = \log C_2+\frac{2}{1-\alpha}(1-\log 2)$. This proves \eqref{al} for $x \in (0, \frac{1}{2C_1} \wedge \frac{(1-\alpha)\lambda^*}{4} \wedge 1)$. The constant $C$ in \eqref{al} can be chosen suitably large so that \eqref{al} continues to hold for $x \in [\frac{1}{2C_1} \wedge \frac{(1-\alpha)\lambda^*}{4}\wedge 1,L]$. This proves the lemma.
\end{proof}

The next lemma quantifies tail probabilities for the distribution of the number of vertices in a ball of radius $r_n$ around the root.
In the following, $|B(v_0,R_n)|$ denotes the number of vertices in $B(v_0,R_n)$.
\begin{lemma}\label{profile}
Suppose $f(i)\leq C_0i^{\alpha}$ for some constant $C_0$ and $\alpha\in (0,1/2]$. Let the sequence $\{R_n\}_{n\geq 1}$ be given by $R_n=2r_n$ where $r_n$ is defined in Lemma \ref{vmaxerr}. Let  $b_n:=\exp\left( \frac{4}{1-\alpha}R_n \log\left( \frac{\log n}{\lambda^*R_n}\right) \right)$. Then there exists $C>0$ such that for all $n\geq 1$,
$$P(|B_n(v_0,R_n)|\geq b_n)\leq C\exp\left(-\frac{f_*}{2}\KK(\frac{1}{\lambda^*}\log n)\right).$$
\end{lemma}
\begin{proof}
Recall the embedding of $\{\GG_l\}_{l \ge 1}$ into a continuous time branching process $\widetilde{BP}_f(\cdot)$ described in Lemma \ref{ctbpem}. Let $F(t):=E[\xi[0,t]]$ and, for $n \ge 0$, let $Z_n(t)$ denote the number of vertices in generation $n$ at time $t$ that are descendants of $v_0$. Equation (1.1) in \cite{Chernoff} gives for $n \ge 0$,
$$F^{n*}(t)=E[Z_n(t)],$$
where $F^{n*}$ is the $n$-fold Stieltjes convolution of the increasing function $F$. For $n=0$ we let $F^{0*}(t)\equiv 1$. It follows from Theorem 1 in \cite{Kingman} that for all $\theta >\underline{\lambda}$,
$$\int_0^\infty e^{-\theta t}dE[Z_n(t)]=\hat\mu(\theta)^n,$$
which gives for any $t >0$, $n \ge 1$,
$$e^{-\theta t}F^{n*}(t)\leq \int_0^t e^{-\theta s}dE[Z_n(s)] \le \hat\mu(\theta)^n.$$
Let $\breve \sigma_n=\frac{1}{\lambda^*}\log n+\gamma_n$ where $\gamma_n=o(\log n)$ and will be specified later. Define $x_n=R_n/\breve \sigma_n$. By \eqref{ktozero}, $x_n \rightarrow 0$ as $n \rightarrow \infty$. By \eqref{lbo}, $\underline{\lambda} :=\inf\{ \lambda>0: \hat{\mu}(\lambda)<\infty\} = 0$. For any $\theta>0$, writing $Z_{\le n}(t)$ for the number of vertices in generation at most $n$ at time $t$ that are descendants of $v_0$,
\beq\label{Znub}
EZ_{\leq \lfloor R_n\rfloor}(\breve \sigma_n) = \sum_{k=0}^{\lfloor R_n \rfloor} F^{k*}(\breve \sigma_n) \leq \sum_{k=0}^{\lfloor R_n \rfloor} e^{\theta \breve \sigma_n} \hat\mu(\theta)^k=e^{\theta \breve\sigma_n} \cdot \frac{\hat\mu(\theta)^{\lfloor R_n\rfloor+1}-1}{\hat\mu(\theta)-1}.
\eeq

Let $\beta:=-\hat\mu'(\lambda^*)$. By the representation (5.40) of \cite{HJarxiv}, $\beta>0$. It is not hard to show that for $x\in (0, \beta^{-1})$,  
$$\theta_x:=\operatorname{arg \ inf}_{0 \le \theta \le \lambda^*}\{ \theta+x \log \hat \mu(\theta)\}\in (0,\lambda^*),$$ 
see the proof of (13.133) in \cite{HJarxiv}. Hence, for $x\in (0,\beta^{-1})$, $\hat\mu(\theta_x)>1$. Obtain $n_0 \in \mathbb{N}$ such that $x_n < \beta^{-1}$ for all $n \ge n_0$. For $n \ge n_0$, taking $\theta=\theta_{x_n}$ in \eqref{Znub} gives
\begin{align}\label{growball}
EZ_{\leq \lfloor R_n \rfloor}(\breve\sigma_n)&\leq e^{\theta_{x_n} \breve\sigma_n}\cdot \frac{\hat\mu(\theta_{x_n})^{ \lfloor R_n \rfloor+1}-1}{\hat\mu(\theta_{x_n})-1}  \leq \frac{\hat\mu(\theta_{x_n})}{\hat\mu(\theta_{x_n})-1} \exp\left(\breve\sigma_n (\theta_{x_n} + x_n \log \hat \mu(\theta_{x_n})\})\right)\nonumber\\
&=c(x_n)\exp(  \tilde\alpha^*(x_n)\breve\sigma_n),
\end{align}
where $c(x):=\frac{\hat\mu(\theta_{x})}{\hat\mu(\theta_{x})-1}$. Note that $\theta_{x_n}\to 0$ as $x_n\rightarrow 0$. Otherwise, we could obtain a subsequence $\{n_j\}$ such that $\liminf_j\theta_{x_{n_j}}>0$. Then, 
$$\liminf_j \tilde\alpha^*(x_{n_j}) = \liminf_j  [\theta_{x_{n_j}} +x \log \hat \mu(\theta_{x_{n_j}})] \ge \liminf_j  [\theta_{x_{n_j}} +x \log \hat \mu(\lambda^*)] = \liminf_j\theta_{x_{n_j}}>0,$$
which is a contradiction to \eqref{al}.
Thus, $\hat\mu(\theta_{x_n})\to \infty$ as $x_n\to 0$, which gives $\lim_{x_n\rightarrow 0} c(x_n)=1$. That is, $C_0:=\sup\{ c(x_n): n \in \mathbb{N}\}<\infty$, which gives
$$EZ_{\leq \lfloor R_n \rfloor}(\breve\sigma_n)\leq C_0\exp(  \tilde\alpha^*(x_n)\breve\sigma_n).$$
Therefore,  
\begin{align*}
P(Z_{\leq \lfloor R_n \rfloor}(T_n) \geq  b_n/2)&\leq P(Z_{\leq \lfloor R_n \rfloor}(T_n)\geq  b_n/2, T_n\leq \breve\sigma_n)+P(T_n> \breve\sigma_n)\\
&\leq P( Z_{\leq \lfloor R_n \rfloor}(\breve\sigma_n)\geq  b_n/2)+P(T_n> \breve\sigma_n)\\
&\leq \frac{2C_0\exp(  \tilde\alpha^*(x_n) \breve\sigma_n)}{b_n}+P(T_n> \breve\sigma_n).
\end{align*}
Taking $\gamma_n=\KK(\frac{1}{\lambda^*}\log n)$, Lemma \ref{Tbounds} (i) with $a= f_*/(2\lambda^*)$ gives an upper bound on the second term,
$$P(T_n> \breve\sigma_n)\leq C  e^{-\frac{f_*}{2} \gamma_n}=C \exp\left(-\frac{f_*}{2}\KK(\frac{1}{\lambda^*}\log n)\right).$$
As $\lim_{l \rightarrow \infty} x_l =0$, therefore, $\sup_lx_l < \infty$. Thus, it follows from Lemma \ref{alphax} upon taking $L= \sup_lx_l$ that there exists some constant $C'>0$ so that for all $n\geq 1$,
\begin{align*}
\exp( \tilde\alpha^*(x_n) \breve\sigma_n)&\leq \exp\left( CR_n+\frac{2}{1-\alpha}R_n\log\left(\frac{\breve\sigma_n}{R_n}\right)\right)\\
&\leq C'\exp\left( \frac{3}{1-\alpha}R_n \log\left( \frac{\log n}{\lambda^*R_n}\right) \right).
\end{align*}
Hence, recalling $b_n=\exp\left( \frac{4}{1-\alpha}R_n \log\left( \frac{\log n}{\lambda^*R_n}\right) \right)$,
\begin{align*}
P(Z_{\leq \lfloor R_n \rfloor}(T_n)\geq  b_n/2)&\leq C \exp\left(-\frac{f_*}{2}\KK(\frac{1}{\lambda^*}\log n)\right)+ 2C_0C' \exp\left( -\frac{1}{1-\alpha}R_n \log\left( \frac{\log n}{\lambda^*R_n}\right) \right)\\
&\leq C''\exp\left(-\frac{f_*}{2}\KK(\frac{1}{\lambda^*}\log n)\right).
\end{align*}
Let $\bar Z_{\leq \lfloor R_n \rfloor}(t)$ denote the number of descendants of $v_1$ by time $t$ that are within generation $\lfloor R_n \rfloor$. It is easy to see $\bar Z_{\leq \lfloor R_n \rfloor}(t)$ has the same law as $Z_{\leq \lfloor R_n \rfloor}(t)$.
The lemma now follows upon noting 
$$P(|B_n(v_0,R_n)|\geq b_n) \le P(Z_{\leq \lfloor R_n \rfloor}(T_n)+\bar Z_{\leq \lfloor R_n \rfloor}(T_n)\geq b_n)\leq 2P(Z_{\leq \lfloor R_n \rfloor}(T_n)\geq b_n/2).$$
\end{proof}

\mn
\textit{Proof of Theorem \ref{nonpersistent}.}
Define the event $\EE^{(n)} := \{|B_n(v_{max}(n),r_n)| \le b_n, v_0 \in B_n(v_{max}(n),r_n)\}$. Then
\begin{align*}
P\left(\EE^{(n)}\right)
&= P\left(|B_n(v_{max}(n),r_n)| \le b_n, v_{max}(n) \in B_n(v_0,r_n)\right)\\
&\geq  P(|B_n(v_0,2r_n)| \leq b_n, v_{max}(n)\in B_n(v_0,r_n)).
\end{align*}
It follows from Lemma \ref{profile} and Lemma \ref{vmaxerr} that there are $a,C_1,C_2,C_3, C, C'>0$, so that 
\begin{align*}
P((\EE^{(n)})^c)&\leq P(|B_n(v_0,2r_n)| \geq b_n)+P( v_{max}(n) \notin B_n(v_0,r_n)))\\
&\leq C_1\exp\left(-\frac{f_*}{2}\KK(\frac{1}{\lambda^*}\log n)\right)+ C_2\exp(-C_3\KK(\frac{1}{\lambda^*}\log n))\\
&\leq C\exp\left( -C'\KK(\frac{1}{\lambda^*}\log n)\right).
\end{align*}
For any $\ep \in (0,1)$, the above bound is $\le \ep$ when $n\geq \exp\left( \lambda^* \KK^{-1}\left(\frac{\log(C/\ep)}{C'}\right)\right)$.
This proves the theorem.
\qed

\bibliographystyle{alpha}
\bibliography{persistence,pref_change_bib,scaling}

\newcommand{\etalchar}[1]{$^{#1}$}
\begin{thebibliography}{CDKM15}

\bibitem[AK68]{athreya1968}
Krishna~B. Athreya and Samuel Karlin.
\newblock Embedding of urn schemes into continuous time {M}arkov branching
  processes and related limit theorems.
\newblock {\em Ann. Math. Statist.}, 39(6):1801--1817, 12 1968.

\bibitem[All10]{Linda10}
Linda~JS Allen.
\newblock {\em An introduction to stochastic processes with applications to
  biology}.
\newblock CRC Press, 2010.

\bibitem[BB21]{BBpersistence}
Sayan Banerjee and Shankar Bhamidi.
\newblock Persistence of hubs in growing random networks.
\newblock {\em Probability Theory and Related Fields}, 180(3-4):891--953, 2021.

\bibitem[BB22]{jordan}
Sayan Banerjee and Shankar Bhamidi.
\newblock Root finding algorithms and persistence of jordan centrality in
  growing random trees.
\newblock {\em The Annals of Applied Probability}, 32(3):2180--2210, 2022.

\bibitem[BBC{\etalchar{+}}12]{borgs2012power}
Christian Borgs, Michael Brautbar, Jennifer Chayes, Sanjeev Khanna, and Brendan
  Lucier.
\newblock The power of local information in social networks.
\newblock In {\em International Workshop on Internet and Network Economics},
  pages 406--419. Springer, 2012.

\bibitem[BBC22]{banerjee2018fluctuation}
Sayan Banerjee, Shankar Bhamidi, and Iain Carmichael.
\newblock Fluctuation bounds for continuous time branching processes and
  nonparametric change point detection in growing networks.
\newblock {\em The Annals of Applied Probability {(to appear)}}, 2022.

\bibitem[BDL17]{bubeck2017findingup}
S{\'e}bastien Bubeck, Luc Devroye, and G{\'a}bor Lugosi.
\newblock Finding {A}dam in random growing trees.
\newblock {\em Random Structures \& Algorithms}, 50(2):158--172, 2017.

\bibitem[BDOC23]{banerjee2023local}
Sayan Banerjee, Prabhanka Deka, and Mariana Olvera-Cravioto.
\newblock Local weak limits for collapsed branching processes with random
  out-degrees.
\newblock {\em arXiv preprint arXiv:2302.00562}, 2023.

\bibitem[BEMR17]{bubeck2017trees}
S{\'e}bastien Bubeck, Ronen Eldan, Elchanan Mossel, and Mikl{\'o}s~Z R{\'a}cz.
\newblock From trees to seeds: on the inference of the seed from large trees in
  the uniform attachment model.
\newblock {\em Bernoulli}, 23(4A):2887--2916, 2017.

\bibitem[BG79]{biggins1979continuity}
John~D Biggins and DR~Grey.
\newblock Continuity of limit random variables in the branching random walk.
\newblock {\em Journal of Applied Probability}, pages 740--749, 1979.

\bibitem[Big77]{Chernoff}
JD~Biggins.
\newblock Chernoff's theorem in the branching random walk.
\newblock {\em Journal of Applied Probability}, 14(3):630--636, 1977.

\bibitem[BK10]{brautbar2010local}
Michael Brautbar and Michael~J Kearns.
\newblock Local algorithms for finding interesting individuals in large
  networks.
\newblock {\em Innovations in Theoretical Computer Science ({ITCS})}, 2010.

\bibitem[BMR15]{bubeck2015influence}
S{\'e}bastien Bubeck, Elchanan Mossel, and Mikl{\'o}s~Z R{\'a}cz.
\newblock On the influence of the seed graph in the preferential attachment
  model.
\newblock {\em IEEE Transactions on Network Science and Engineering},
  2(1):30--39, 2015.

\bibitem[BRST01]{bollobas2001degree}
B{\'e}la Bollob{\'a}s, Oliver Riordan, Joel Spencer, and G{\'a}bor Tusn{\'a}dy.
\newblock The degree sequence of a scale-free random graph process.
\newblock {\em Random Structures \& Algorithms}, 18(3):279--290, 2001.

\bibitem[BV14]{boldi2014axioms}
Paolo Boldi and Sebastiano Vigna.
\newblock Axioms for centrality.
\newblock {\em Internet Mathematics}, 10(3-4):222--262, 2014.

\bibitem[CDKM15]{curien2015scaling}
Nicolas Curien, Thomas Duquesne, Igor Kortchemski, and Ioan Manolescu.
\newblock Scaling limits and influence of the seed graph in preferential
  attachment trees.
\newblock {\em Journal de l’{\'E}cole polytechnique-Math{\'e}matiques},
  2:1--34, 2015.

\bibitem[DM09]{SPdegree}
Steffen Dereich and Peter M{\"o}rters.
\newblock Random networks with sublinear preferential attachment: degree
  evolutions.
\newblock {\em Electronic Journal of Probability}, 14:1222--1267, 2009.

\bibitem[DR18]{devroye2018discovery}
Luc Devroye and Tommy Reddad.
\newblock On the discovery of the seed in uniform attachment trees.
\newblock {\em arXiv preprint arXiv:1810.00969}, 2018.

\bibitem[FP17]{frieze2017looking}
Alan Frieze and Wesley Pegden.
\newblock Looking for vertex number one.
\newblock {\em Annals of Applied Probability}, 27(1):582--630, 2017.

\bibitem[Fre77]{freeman1977set}
Linton~C Freeman.
\newblock A set of measures of centrality based on betweenness.
\newblock {\em Sociometry}, pages 35--41, 1977.

\bibitem[Gal13]{galashin2013existence}
Pavel Galashin.
\newblock Existence of a persistent hub in the convex preferential attachment
  model.
\newblock {\em arXiv preprint arXiv:1310.7513}, 2013.

\bibitem[GvdH18]{garavaglia2018trees}
Alessandro Garavaglia and Remco van~der Hofstad.
\newblock From trees to graphs: collapsing continuous-time branching processes.
\newblock {\em Journal of Applied Probability}, 55(3):900--919, 2018.

\bibitem[HJ16]{HJarxiv}
C~Holmgren and S~Janson.
\newblock Fringe trees, crump--mode--jagers branching processes and m-ary
  search trees. preprint.
\newblock {\em arXiv preprint arxiv:1601.03691}, 2016.

\bibitem[JL16]{jog2016analysis}
Varun Jog and Po-Ling Loh.
\newblock Analysis of centrality in sublinear preferential attachment trees via
  the {C}rump-{M}ode-{J}agers branching process.
\newblock {\em IEEE Transactions on Network Science and Engineering},
  4(1):1--12, 2016.

\bibitem[JL18]{jog2018persistence}
Varun Jog and Po-Ling Loh.
\newblock Persistence of centrality in random growing trees.
\newblock {\em Random Structures \& Algorithms}, 52(1):136--157, 2018.

\bibitem[JN84]{jagers1984growth}
Peter Jagers and Olle Nerman.
\newblock The growth and composition of branching populations.
\newblock {\em Advances in Applied Probability}, 16(2):221--259, 1984.

\bibitem[Kin75]{Kingman}
John Frank~Charles Kingman.
\newblock The first birth problem for an age-dependent branching process.
\newblock {\em The Annals of Probability}, 3(5):790--801, 1975.

\bibitem[KL16]{khim2016confidence}
Justin Khim and Po-Ling Loh.
\newblock Confidence sets for the source of a diffusion in regular trees.
\newblock {\em IEEE Transactions on Network Science and Engineering},
  4(1):27--40, 2016.

\bibitem[Liu01]{Liu2001}
Quansheng Liu.
\newblock Asymptotic properties and absolute continuity of laws stable by
  random weighted mean.
\newblock {\em Stochastic processes and their applications}, 95(1):83--107,
  2001.

\bibitem[LP19]{lugosi2019finding}
Gabor Lugosi and Alan~S Pereira.
\newblock Finding the seed of uniform attachment trees.
\newblock {\em Electronic Journal of Probability}, 24, 2019.

\bibitem[MR19]{Mori19}
Tam{\'a}s~F M{\'o}ri and S{\'a}ndor Rokob.
\newblock Moments of general time dependent branching processes with
  applications.
\newblock {\em Acta Mathematica Hungarica}, 159(1):131--149, 2019.

\bibitem[Ner81]{nerman1981convergence}
Olle Nerman.
\newblock On the convergence of supercritical general ({CMJ}) branching
  processes.
\newblock {\em Probability Theory and Related Fields}, 57(3):365--395, 1981.

\bibitem[New05]{newman2005measure}
Mark~E.J. Newman.
\newblock A measure of betweenness centrality based on random walks.
\newblock {\em Social {N}etworks}, 27(1):39--54, 2005.

\bibitem[NK11]{navlakha2011network}
Saket Navlakha and Carl Kingsford.
\newblock Network archaeology: uncovering ancient networks from present-day
  interactions.
\newblock {\em PLoS computational biology}, 7(4), 2011.

\bibitem[PRR17]{pekoz2017joint}
Erol Pek{\"o}z, Adrian R{\"o}llin, and Nathan Ross.
\newblock Joint degree distributions of preferential attachment random graphs.
\newblock {\em Advances in Applied Probability}, 49(2):368--387, 2017.

\bibitem[RTV07]{rudas2007random}
Anna Rudas, B{\'a}lint T{\'o}th, and Benedek Valk{\'o}.
\newblock Random trees and general branching processes.
\newblock {\em Random Structures \& Algorithms}, 31(2):186--202, 2007.

\bibitem[S{\'e}n21]{senizergues2021geometry}
Delphin S{\'e}nizergues.
\newblock Geometry of weighted recursive and affine preferential attachment
  trees.
\newblock {\em Electronic Journal of Probability}, 26:1--56, 2021.

\bibitem[SZ11]{shah2011rumors}
Devavrat Shah and Tauhid Zaman.
\newblock Rumors in a {N}etwork: Who's the {C}ulprit?
\newblock {\em IEEE Transactions on {I}nformation {T}heory}, 57(8):5163--5181,
  2011.

\bibitem[SZ12]{shah2012rumor}
Devavrat Shah and Tauhid Zaman.
\newblock Rumor centrality: a universal source detector.
\newblock In {\em Proceedings of the 12th ACM SIGMETRICS/PERFORMANCE joint
  international conference on Measurement and Modeling of Computer Systems},
  pages 199--210, 2012.

\end{thebibliography}

\begin{acks}
SB was supported in part by the NSF-CAREER award (DMS-2141621) and the NSF RTG award (DMS-2134107).
We are grateful to Shankar Bhamidi for continued discussions throughout the project. 
We thank Pratima Hebbar for very helpful advice.
We also thank an anonymous referee and an associate editor whose suggestions substantially improved this article.
\end{acks}

%%%%%%%%%%%%%%%%%%%%%%%%%%%%%%%%%%%%%%%%%%%%%%%%%%%%%%%%%%%%%%%%%%%
%%                                                               %%
%% You have reached the end of your document.                    %%
%%                                                               %%
%%%%%%%%%%%%%%%%%%%%%%%%%%%%%%%%%%%%%%%%%%%%%%%%%%%%%%%%%%%%%%%%%%%

\end{document}